\documentclass[a4paper,twoside,11pt,preprint]{article}

\usepackage[english]{babel}
\usepackage[ansinew]{inputenc}
\usepackage{geometry}
\usepackage{color} 
 \usepackage{amsmath}
 \numberwithin{equation}{section}
\usepackage{overpic}
\usepackage{hyperref}
\usepackage{graphicx}

\usepackage{amsmath}
\usepackage{amsthm}
\usepackage{amsfonts}
\usepackage{amssymb}

\newtheorem{theorem}{Theorem}
\newtheorem{proposition}[theorem]{Proposition}
\newtheorem{lemma}[theorem]{Lemma}
\newtheorem{corollary}[theorem]{Corollary}
\newtheorem{satz}{Theorem}

\theoremstyle{remark}
\newtheorem{definition}{Definition}
\newtheorem{remark}{Remark}
\newtheorem{example}{Example}

\definecolor{orange}{rgb}{1,0.5,0}

\renewcommand{\Im}{\operatorname{Im}}
\renewcommand{\Re}{\operatorname{Re}}

\newcommand{\intd}{\mathrm{d}} 

\newcommand{\lmr}{\operatorname{lmr}} 
\newcommand{\hcap}{\operatorname{hcap}}

\newcommand{\lt}{\underline{t}} 
\newcommand{\gt}{\overline{t}} 
\newcommand{\ltau}{\underline{\tau}} 
\newcommand{\gtau}{\overline{\tau}} 

\newcommand{\lenpar}{0.2\baselineskip} 
\newcommand{\scalefactor}{0.9} 

\newcommand{\C}{\mathbb{C}}
\newcommand{\D}{\mathbb{D}}
\newcommand{\R}{\mathbb{R}}
\newcommand{\N}{\mathbb{N}}
\newcommand{\Ha}{\mathbb{H}}

\newcommand{\LandauO}{\mathcal{O}} 

\newcommand{\eps}{\varepsilon}
\newcommand{\diam}{\operatorname{diam}}

\newcounter{cntselflist}
\newenvironment{selflist}{\begin{list}{\mbox{}\hspace{0.5cm}\mbox{}\arabic{cntselflist})\mbox{ }}{
	\setlength\itemindent{0pt}
	\setlength\leftmargin{0pt}
	\setlength\itemsep{3pt}
	\setlength\labelwidth{0pt}
	\setlength\labelsep{0pt}
	\usecounter{cntselflist}
	}  }{\end{list}}

\begin{document}
\parindent 0pt 

\setcounter{section}{0}

\author{Christoph B\"ohm, Sebastian Schlei\ss inger\thanks{Partially supported by the ERC grant  ``HEVO - Holomorphic Evolution Equations'' no. 277691.}}
 \title{
 The Loewner equation for multiple slits, multiply connected domains and branch points}

\date{\today}
\maketitle

\begin{abstract}
    Let $\gamma_1,\gamma_2:[0,T]\to \overline{\D}\setminus\{0\}$ be parametrizations
	of two slits $\Gamma_1:=\gamma(0,T], \Gamma_2=\gamma_2(0,T]$ such that 
	$\Gamma_1$ and $\Gamma_2$ are disjoint. \\
	Let $g_t$ to be the unique normalized conformal mapping from 
	$\D\setminus (\gamma_1[0,t]\cup \gamma_2[0,t])$ onto $\D$ with
	$g_t(0)=0,$ $g'_t(0)>0$. Furthermore, for $k=1,2$, denote by $h_{k;t}$ 
	the unique normalized conformal mapping from 
	$\D\setminus \gamma_k[0,t]$ onto $\D$ with
	$h_{k;t}(0)=0,$ ${h'_{k;t}(0)}>0$.\\
	Loewner's famous theorem (\cite{Loewner:1923}) can be stated in the following way: 
	The function $t\mapsto h_{k;t}$
	is differentiable at $t_0$ if and only if $t\mapsto \log(h_{k;t}'(0))$ is differentiable at 
	$t_0$.\\
	In this paper we compare the differentiability of $t\mapsto h_{k;t}$ with that of 
	$t\mapsto g_t.$ We show that the situation is more complicated
	in the case $t_0=0$ with $\gamma_1(0)=\gamma_2(0).$\\
	Furthermore, we also look at this problem in the case of a multiply connected domain with 
	its corresponding Komatu-Loewner equation.\\
	\end{abstract}

	\maketitle

	
\section{Introduction and results}

\subsection{The main results}

\textbf{The simply connected case}\\

By $\D:=\{z\in\C \,|\, |z|<1\}$ we denote the unit disk.\\
	
	Let $\gamma\colon[0,T]\to \overline{\D}$ be a simple curve (i.e. $\gamma$ is continuous and 
	injective) with $\Gamma:=\gamma(0,T]\subset \D\setminus \{0\}$
	and $\gamma(0)\in\partial \D.$ In the following such a set $\Gamma$ will be called
	\emph{slit}.\\
	
	For every $t\in[0,T]$, the domain $\Omega_t:=\D\setminus \gamma[0,t]$
	is  simply connected and it can be mapped onto $\D$ by a conformal map
	$g_t:\Omega_t\to \D.$\\
	This mapping is unique if we require the normalization $g_t(0)=0,$ $g'_t(0)>0$.	
	The function $g_t'(0)$ is increasing and $g_t'(0)\geq 1$ for all $t$ as a consequence of
	the Schwarz lemma. The \emph{logarithmic mapping radius} is defined as
	$\lmr(g_t):=\log(g_t'(0)).$\\
	
	In his much celebrated paper from 1923 (\cite{Loewner:1923}), Loewner considered
	the question whether the function $t\mapsto g_t$ could be differentiable, even though 
	there are no smoothness assumptions on $\Gamma.$
	Loewner's famous theorem can be stated in the following way: \\
	The differentiability of $t\mapsto \lmr(g_t)$ is equivalent to the 
	differentiability of the function
	$t\mapsto g_t,$	more precisely the following statement holds;
	see, e.g., Theorem 2 in \cite{BoehmLauf}.
	\begin{satz}\label{Loewner_eq}
	The function $c(t):=\lmr(g_t)$ is differentiable at $t=t_0$ if and only if the
	family $\{g_t\}_{t\in[0,T]}$ 
	is differentiable at $t=t_0$, i.e. for every $z\in \D\setminus \Gamma,$ the function 
	$t\mapsto g_t(z)$ is differentiable at $t=t_0.$ In this case, $g_t(z)$ satisfies the following 
	differential equation:
	\begin{equation}
         \dot{g}_{t_0}(z) = \dot{c}(t_0)\cdot g_{t_0}(z) \cdot \frac{\xi(t_0)+g_{t_0}(z)}{\xi(t_0)-g_{t_0}(z)},
	\end{equation}
	where $\xi(t_0)= \lim_{z\to \gamma(t_0)} g_{t_0}(z).$
	\end{satz}
	In the following, we will call $\gamma$ a \emph{$\D$-Loewner parametrization for $\Gamma$ at $t_0$}, if 
	the two equivalent conditions in Theorem \ref{Loewner_eq} hold.
	\begin{remark}
	 Usually, the parametrization of $\Gamma$ is chosen in such a way that $\lmr(g_t)=t.$ In this case, the 
	 mappings $\{g_t\}$ are (continuously) differentiable for all $t\in[0,T].$ Thus,
	 an arbitrary slit $\Gamma$ can be described by a differential equation for 
	 the family $\{g_t\}$. This celebrated idea of Loewner turned
	 out to be quite useful for the	 theory of univalent mappings and its most prominent
	 application nowadays is the stochastic Loewner evolution invented by Schramm in 2000.	 
	\end{remark}
	
	Now let $\gamma_1,\gamma_2\colon[0,T]\to \overline{\D}\setminus\{0\}$ be parametrizations
	of two slits $\Gamma_1:=\gamma_1(0,T], \Gamma_2:=\gamma_2(0,T]$ such that 
	$\Gamma_1$ and $\Gamma_2$ are disjoint. \\
	For a fixed time $t_0$ we will distinguish 
	between two cases:\\
	Either $\gamma_1(t_0)\not =\gamma(t_0)$ (``\emph{disjoint case}'') or 
	$\gamma_1(t_0)=\gamma_2(t_0)$ (``\emph{branch point case}''), which is only possible 
	for $t_0=0$.\\
	
	Again, we can define $g_t$ to be the unique normalized conformal mapping from 
	$\Omega_t:=\D\setminus (\gamma_1[0,t]\cup \gamma_2[0,t])$ onto $\D$ with
	$g_t(0)=0,$ $g'_t(0)>0$.\\
	We are interested in the question, under which conditions the family $\{g_t\}_{t\in[0,T]}$ 
	is differentiable at a point $t_0\in[0,T]$. \\
	Again, a necessary condition is that $c(t):=\lmr(g_t):=\log(g'_t(0))$
	is differentiable at $t=t_0.$ However, this condition is not sufficient anymore,
	see Example \ref{Exa-Counterexample2}.\\
	On the other hand, the two statements are equivalent in the branch point case; see Theorem
	\ref{Simply_multiply}.\\

	In the disjoint case, differentiability of $t\mapsto g_t$
	is guaranteed if both slits are $\D$-Loewner 
	parametrized. More precisely, the following equivalence holds.  
	
	\begin{theorem} \label{The-KLEvsLE_Intro}
	Suppose that $t_0\in[0,T]$ such that $\gamma_1(t_0)\not=\gamma_2(t_0).$  
	Then the following two conditions are equivalent:
	\begin{enumerate}
		\item For $j=1,2$, $\gamma_j$ is a $\D$-Loewner parametrization for $\Gamma_j$ at $t_0.$
		\item The function $t\mapsto g_t(z)$ is differentiable at $t_0$ for every 
			$z\in \Omega_{t_0}$. 
	\end{enumerate}

\end{theorem}
For $j=1,2$, let $h_{j;t}$ be the unique conformal mapping from $\D\setminus \gamma_j[0,t]$ 
onto $\D$ with $h_{j;t}(0)=0,$ $h'_{j;t}(0)>0$ and
let $c_j(t):=\lmr(h_{j;t})=\log(h'_{j;t}(0)).$ We will also derive a relation between
$\dot{c}$ and $\dot{c}_j.$ Here we note the simplest case $t_0=0:$\\
If the two equivalent statements in Theorem \ref{The-KLEvsLE_Intro} hold for $t_0=0,$ 
then $c(t)$ is differentiable at $t=0$ with
\begin{equation}\label{disjoint_sum}
\dot{c}(0)=\dot{c}_1(0)+\dot{c}_2(0). 
\end{equation}
 A general relation between $\dot c$ and $\dot c_j$ if $t_0>0$ is given by Theorem \ref{The-KLEvsLE}.

The situation is different for the branch point case.

	\begin{theorem}\label{counterexample}
	There exist two slits $\Gamma_1, \Gamma_2$ in 
	$\D$ with $\overline{\Gamma_1}\cap \overline{\Gamma_2}=\{p\}\subset \partial\D$ with $\D$-Loewner parametrizations $\gamma_k\colon[0,T]\to\Gamma_k$ in
	$[0,T],$
	such that the function $t\mapsto g_t(z),$ $z\in \Omega_T$, is not differentiable at $t=0.$ 
	\end{theorem}
	On the other hand, we also give a condition ensuring differentiability of
	$t\mapsto g_t(z)$ at $t=0$ in this case.
	
	\begin{definition}Let $\alpha\in(0,\pi).$ We say that a simple curve $\gamma\colon[0,T]\to \overline{\D},$ $\gamma(0)\in\partial \D,$
	$\gamma(0,T]\subset \D,$ \emph{approaches $\partial\D$ in $\alpha-$direction}, (see Figure 1)
	if for every $\eps>0$	there exists $s>0$ such that
	$$\gamma(0,s]\subset \{z\in \D \,|\, \alpha-\eps < \arg(\gamma(0)-z)+\arg(\gamma(0))
	-\pi/2 < \alpha + \eps\}.$$
	\end{definition}

	\begin{figure}[h]\label{app} 
\centering 
	\begin{overpic}[scale=0.33]
		{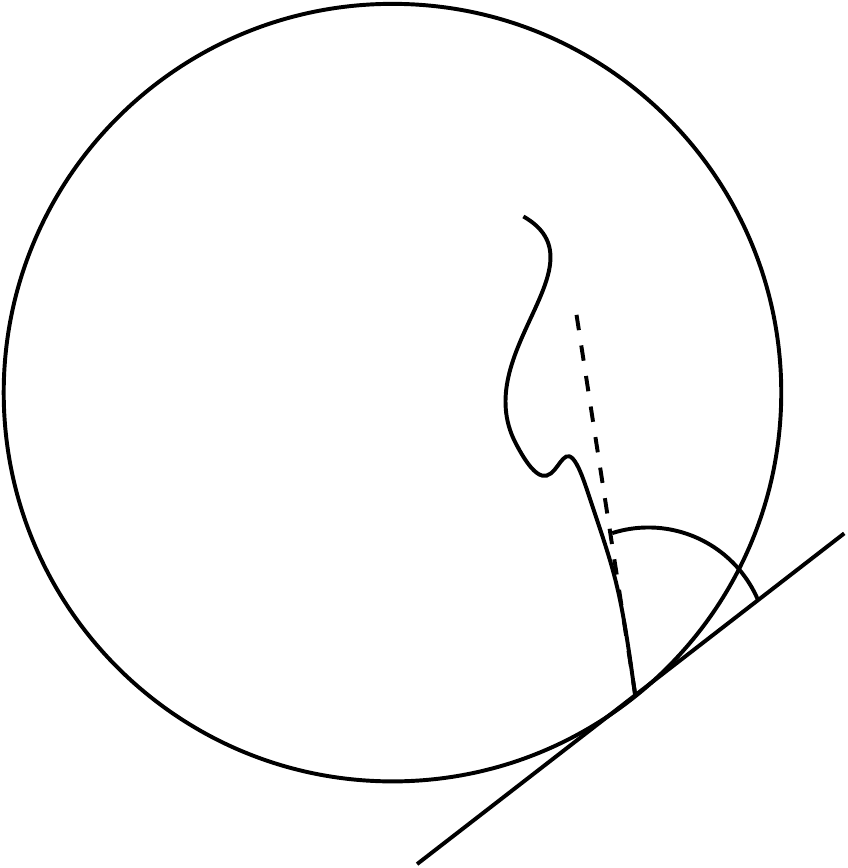}
		\put(74,29){$\alpha$}
	\end{overpic}
	\caption{A slit \emph{approaching $\partial\D$ in $\alpha-$direction}.}
\end{figure}

	\begin{theorem}\label{lines}
	Let $b_1,b_2\geq 0$, $\gamma_1(0)=\gamma_2(0)$ and assume that $\Gamma_j$
	approaches $\partial\D$ in $\alpha_j$-direction
	with $\alpha_1\leq \alpha_2$. Let $\gamma_j$ be a $\D$-Loewner parametrization for $\Gamma_j$
	at $t=0$ for $j=1$ and $j=2$ with $b_1=\dot{c}_1(0), b_2=\dot{c}_2(0).$ Then 
	the function $t\mapsto g_t(z)$ is differentiable at $t=0$ for every 
			$z\in \Omega.$
	\begin{itemize} 
	\item If $b_1=0$ or $b_2=0,$ then $\dot{c}(0)=\max\{b_1,b_2\}.$
	\end{itemize}
	If $b_1, b_2>0,$ then 
	\begin{itemize}
	 \item  $\max\{b_1,b_2\}	\leq  \dot{c}(0) < b_1+b_2,$
	 \item $\dot{c}(0) =  \max\{b_1,b_2\} \quad \text{if and only if} \quad\alpha_1=\alpha_2,  \quad \text{and}$
	 \item $\dot{c}(0)\to b_1+b_2 \quad \text{as}\quad (\alpha_1,\alpha_2)\to (0,\pi).$
	\end{itemize}
	
	\end{theorem}
	
Note that the very last statement says that the branch point case behaves like the disjoint case 
when $(\alpha_1, \alpha_2)\to(0,\pi),$ see equation \eqref{disjoint_sum}.\\

Finally it is worth mentioning that the converse of Theorem \ref{lines} is wrong;
see Example \ref{Exa-Counterexample3}.\\

\textbf{The multiply connected case}\\

A \emph{circular slit disk} $D$ is an $n$-connected domain of the form 
	$D = \D \setminus (C_1\cup...\cup C_{n-1})$,
	where the $C_j$'s are proper disjoint circular arcs in $\D$ centered at $0$. For any circular slit disk $D$
	and any $u\in\partial \D$, we denote by $w\mapsto \Phi(u,w;D)$
	the unique conformal mapping from $D$ onto the right half-plane minus slits parallel
	to the imaginary axis with $\Phi(u,u;D)=\infty$ and $\Phi(u,0;D)=1.$ For example, 
	$\Phi(u,w;\D)=\frac{u+w}{u-w}.$\\
	
	Now let $\Omega $ be an $n$-connected circular slit disk and let $\gamma\colon[0,T]\to \overline{\D}$
	be a simple curve with $\Gamma:=\gamma(0,T]\subset \Omega\setminus \{0\}$
	and $\gamma(0)\in\partial \D.$ In this case, $\Omega_t:=\Omega\setminus \gamma[0,t]$ is an $n$-connected domain for every $t\in[0,T]$
	and it can be mapped onto a circular slit disk $D_t$ by a conformal map $g_t:\Omega_t\to D_t.$
	This mapping is unique if we require the normalization $g_t(0)=0,$ $g'_t(0)>0$, 
	$g_t(\partial\D)\subset\partial\D$; see \cite{ConwayII}, Chapter 15.6. In the following, we will call
	mappings \emph{normalized} if they satisfy these three conditions.\\

	Again we define the \emph{logarithmic mapping radius}
	$\lmr(g_t):=\log(g_t'(0)).$
	The analog of Theorem \ref{Loewner_eq} is given by the following Theorem; see
	Theorem 5.1 in \cite{BauerFriedrichCSD} or  Theorem 2 in \cite{BoehmLauf}. Loewner equations for  multiply
	connected domains were first studied by Komatu; see \cite{Komatu}, \cite{KomatuZweifach}.
	\begin{satz}\label{LoewnerKomatu_eq} The function $c(t):=\lmr(g_t)$ is differentiable at $t=t_0$ if and only if the
	family $\{g_t\}_{t\in[0,T]}$ 
	is differentiable at $t=t_0$, i.e. for every $z\in \Omega\setminus \Gamma,$ the function 
	$t\mapsto g_t(z)$ is differentiable at $t=t_0.$ In this case, $g_t(z)$ satisfies the following 
	differential equation:
	\begin{equation}
         \dot{g}_{t_0}(z) = \dot{c}(t_0)\cdot g_{t_0}(z) \cdot \Phi(\xi(t_0),g_{t_0}(z);D_{t_0}),
	\end{equation}
	where $\xi(t_0)= \lim_{z\to \gamma(t_0)} g_{t_0}(z).$
	\end{satz}
	In the following, we will call $\gamma$ an \emph{$\Omega$-Loewner parametrization for $\Gamma$ at $t_0$}, if 
	the two equivalent conditions in Theorem \ref{LoewnerKomatu_eq} hold.\\
	
	The following relation to $\D$-Loewner parametrizations is not surprising.
	\begin{theorem}\label{different_para}
	 Let $t_0\in [0,T.]$ Then $\gamma$ is an $\Omega$-Loewner parametrization for $\Gamma$ at $t_0$ if and only if it is
	 a $\D$-Loewner parametrization for $\Gamma$ at $t_0$.
	\end{theorem}

	Now we pass again to the case of two slits: Let $\gamma_1, \gamma_2\colon[0,T]\to \overline{\D}$ be parametrizations
	of two slits $\Gamma_1=\gamma(0,T]$ and $\Gamma_2=\gamma_2(0,T]$ such that
	$\Gamma_1$ and $\Gamma_2$ are disjoint, $\Gamma_1, \Gamma_2\subset \Omega\setminus\{0\}$ and
	$\gamma_1(0)\not=\gamma_2(0)$ or $\gamma_1(0)=\gamma_2(0)$.\\
		
	Again, we define $g_t$ to be the unique normalized mapping from 
	$\Omega_t:=\Omega\setminus (\gamma_1[0,t]\cup \gamma_2[0,t])$ onto a circular slit disk $\D_t$ and $\lmr(g_t):=\log(g'_t(0))$.\\
	Furthermore, let $h_t$ be the unique normalized mapping from
	$\Psi_{t}:=\D\setminus (\gamma_1[0,t]\cup \gamma_2[0,t])$ onto $\D.$
	\begin{theorem}\label{Simply_multiply}
	 Let $t_0\in [0,T].$ Then the following two statements are equivalent. 
	 \begin{enumerate}
	  \item The function $t\mapsto g_t(z)$ is differentiable at $t_0$ for every 
			$z\in \Omega_{t_0}$.
	  \item The function $t\mapsto h_t(z)$ is differentiable at $t_0$ for every 
			$z\in \Psi_{t_0}$.
	 \end{enumerate}
	 In the branch point case, i.e. $\gamma_1(0)=\gamma_2(0)$ and $t_0=0,$
	 the above statements are equivalent to each of the following two statements.
	  \begin{enumerate}
	 	\item[3.] The function $t\mapsto\lmr(h_t)$ is differentiable at $0$.
		\item[4.] The function $t\mapsto\lmr(g_t)$ is differentiable at $0$.
	 \end{enumerate}		
	\end{theorem}
	As a direct consequence of the last two theorems, we can state Theorem \ref{The-KLEvsLE_Intro}
	and Theorem \ref{lines} for the multiply connected case. 
  
	\begin{corollary}\label{reduce}
	Suppose that $t_0\in[0,T]$ such that $\gamma_1(t_0)\not=\gamma_2(t_0).$
	Then the following conditions are equivalent:
	\begin{enumerate}
	\item The function $t\mapsto g_t(z)$ is differentiable at $t_0$ for every 
			$z\in \Omega_{t_0}$. 
	\item For $j=1,2$, $\gamma_j$ is a $\D$-Loewner parametrization for $\Gamma_j$ at $t_0.$

	\end{enumerate}

	\end{corollary}
	Corollary \ref{reduce} shows that the question whether the function $t\mapsto g_t(z)$ 
	is differentiable at $t_0$ can be reduced to the corresponding question for each single slit with respect to the 
	simply connected domain $\D.$
	\begin{corollary}
	Suppose $\gamma_1(0)=\gamma_2(0)$ and that $\Gamma_j$ approaches $\partial\D$ in $\alpha_j$-direction
	with $\alpha_1\leq \alpha_2$. If $\gamma_j$ is a $\D$-Loewner parametrization for $\Gamma_j$
	at $t=0$ for $j=1$ and $j=2$, then the function $t\mapsto g_t(z)$ is differentiable at $t=0$ for every 
			$z\in \Omega.$
	\end{corollary}
	\begin{remark}
	 All statements presented here can be easily generalized to the case of $m>2$ slits and to slits
	 that are branched within the unit disc.\\
	 We only consider the case of two slits and of one branch point on $\partial \D$ in order to 
	 simplify the notation in the proofs.
	\end{remark}

	\subsection{Organization of the paper}\label{organization}
	
	Before we pass on to the proofs of Theorems 1-5, we will explain how Theorems \ref{The-KLEvsLE_Intro}, 
	\ref{different_para} and \ref{Simply_multiply} follow from a more technical statement. To this end,
	we first introduce some further notations.\\
	
We denote by $\Omega$ an arbitrary circular slit disk.\\
Let $m=1$ or $m=2$ and let $\gamma_1, ..., \gamma_m\colon[0,T]\rightarrow \bar\Omega\setminus\{0\}$
be Jordan arcs with $\gamma_k(0)\in\partial\D$ and $\Gamma_k:=\gamma_k(0,T]\subset\Omega.$ In case $m=2$, we suppose that $\Gamma_1$ and $\Gamma_2$ are disjoint.\\
The normalized conformal mapping $g_t$ is defined as before, i.e. 
$g_t$ maps $\Omega_t:=\Omega\setminus\bigcup_{k=1}^m
\gamma_k[0,t]$ onto the circular slit disk $D_t$.\\
To simplify the notation, we will also write $\Phi(\xi,z;t)$ instead of $\Phi(\xi,z;D_t).$
\begin{center}
\begin{overpic}[scale=\scalefactor]
	{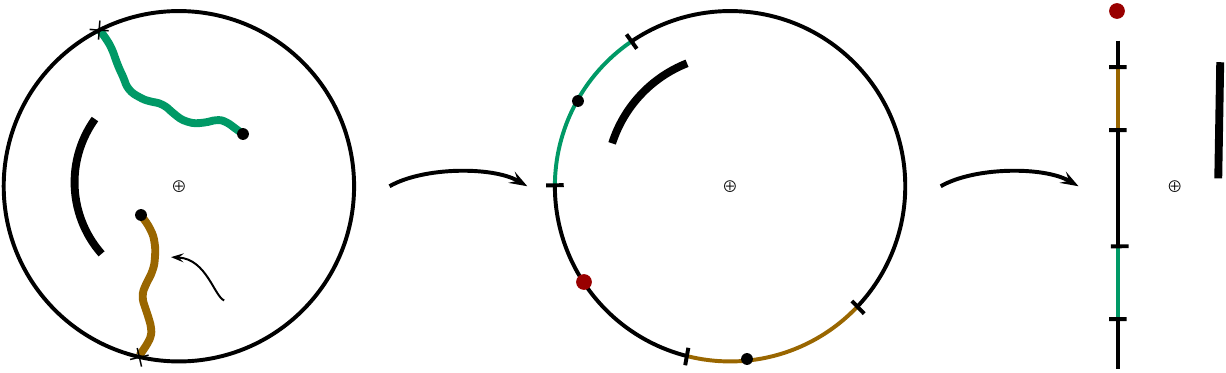}
	\put(14,24){$\Omega_t$}
	\put(18.3,5){$k$}
	\put(12,13){$t$}
	\put(20.4,19){$t$}
	\put(36,18){$g_t$}
	\put(52,10){$D_t$}
	\put(75.5,18){$\Phi(\xi,z;t)$}
	\put(45,6){$\xi$}
	\put(59,3){$\xi_k(t)$}
	\put(92,28){$\infty$}
	\put(15.3,15){$0$}
	\put(60.2,15){$0$}
	\put(96.2,13){$1$}
\end{overpic}
\end{center}
Beside $\Omega_t$, we set $\Delta_{k}(t):=\D\setminus \gamma_k[0,t]$. Note that $\Delta_k(t)$ is 
simply connected, whereas $\Omega_t$ is $n$-connected. 
As before, we denote by $h_{k;t}\colon\Delta_{k}(t)\rightarrow\D$ the unique conformal
mapping with the normalization $h_t(0)=0$ and $h'_t(0)>0$.
\begin{center}
\begin{overpic}[scale=\scalefactor]
	{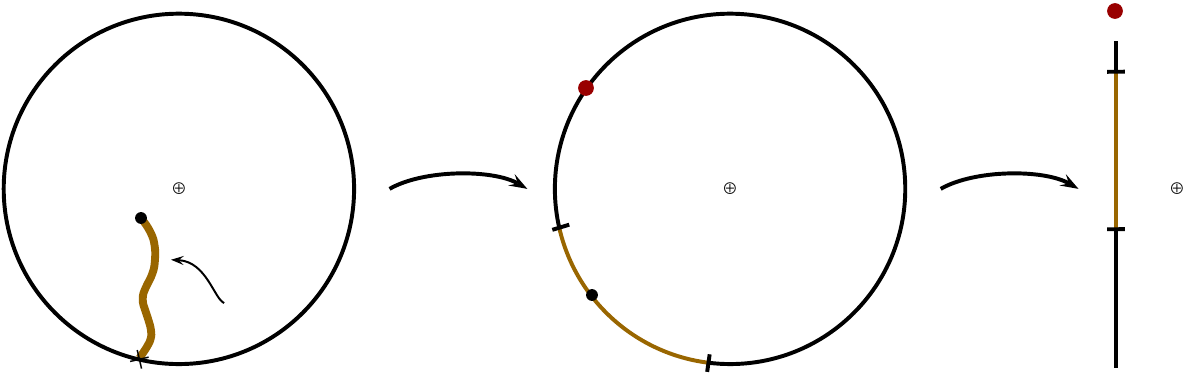}
	\put(14,24){$\Delta_t$}
	\put(19,5){$k$}
	\put(10,13){$t$}
	\put(35,18.3){$h_{k;t}$}
	\put(52,10){$\D$}
	\put(82,19.5){$\frac{\zeta+z}{\zeta-z}$}
	\put(46.6,23.4){$\zeta$}
	\put(42,4){$\zeta_k(t)$}
	\put(94.7,29){$\infty$}
	\put(16,15){$0$}
	\put(62.4,15){$0$}
	\put(99.3,13){$1$}
\end{overpic}
\end{center}
Moreover, we will make use of the driving functions of $g_t$ and $h_{k;t}$ defined by $\xi_k(t):=g_t(\gamma_k(t))$ and 
$\zeta_k(t):=h_{k;t}(\gamma_k(t)),$ respectively, for all $t\in[0,T]$ and all $k=1,\ldots,m.$ 

\begin{remark}\label{rm:continuous}
We note that the driving functions $\xi_k, \zeta_k\colon[0,T]\to\partial\D$ are continuous by Proposition 8
from \cite{BoehmLauf}.
\end{remark}

In order to give a connection between differentiability of $t\mapsto g_t(z)$
and $t\mapsto h_{k;t}(z)$ we need one further abbreviation. Therefore we set
\[
	\alpha_k(t) := \left|\frac{\intd}{\intd z} (g_{t}\circ h^{-1}_{k;t})(z)\big|_{z=\zeta_k(t)}
		\right|
\]
for all $t\in[0,T]$ and all $k=1,\ldots,m$. The derivative is well-defined, as 
$g_{t}\circ h^{-1}_{k,t}$ can be extended by the Schwarz refection principle to an analytic
function at $z=\zeta_k(t)$.
\begin{center}
\begin{overpic}[scale=\scalefactor]
	{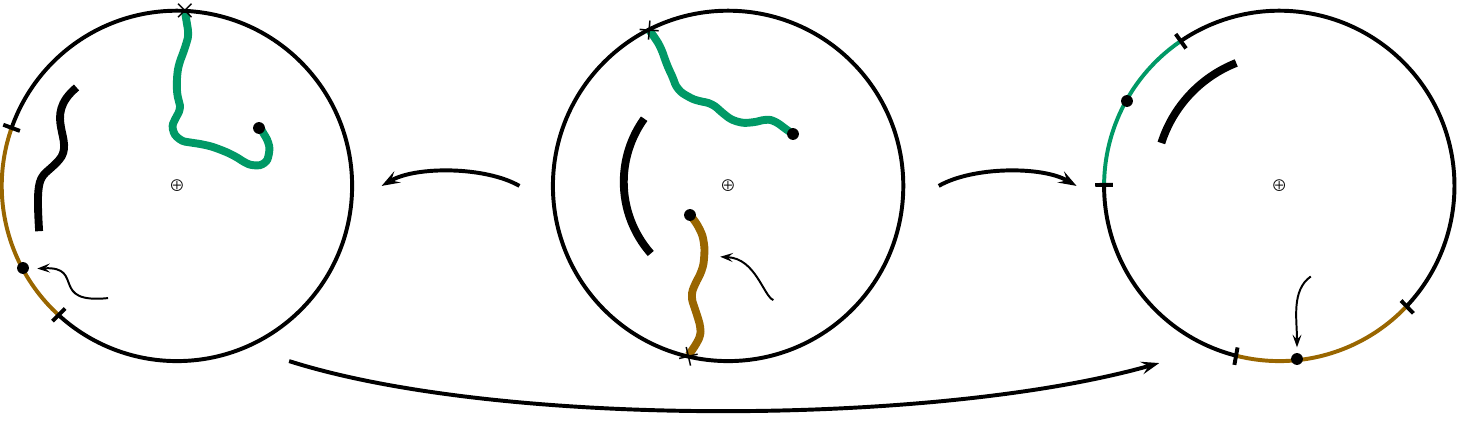}
		\put(8,8){$\zeta_k(t)$}
		\put(29,19){$h_{k;t}$}
		\put(49,24){$\Omega_t$}
		\put(54.5,21){$t$}
		\put(46.5,15.5){$t$}
		\put(53,8.5){$k$}
		\put(68,19){$g_t$}
		\put(90,20){$D_t$}
		\put(90.5,11){$\xi_k(t)$}
\end{overpic}
\end{center}
Note that $\alpha_k(t)\le 1$ holds for all $t\in[0,T]$ if $\Omega$ is simply connected, 
i.e. if $\Omega=\D$, see Lemma \ref{Lem-LambdaSimplyConnectedCase}.
Then we find the following theorem.
\begin{theorem} \label{The-KLEvsLE}
	Let $t_0\in[0,T]$ with $\gamma_1(t_0)\not=\gamma_2(t_0).$\\ Let $z_0\in\Omega_{t_0}\setminus\{0\}$, 
	then the following two conditions are equivalent.
	\begin{enumerate}
		\item Each function $t\mapsto h_{k;t}(z_0)$ is differentiable at $t_0$ for every 
			$k=1,\ldots,m$.
		\item The function $t\mapsto g_t(z)$ is differentiable at $t_0$ for every 
			$z\in \Omega_{t_0}$. 
	\end{enumerate}
	If $t\mapsto g_t(z)$ is differentiable at $t_0$ for every $z\in \Omega_{t_0}$, then 
	\begin{align} \label{Equ:KLE} \tag{$\star$}
		\dot{g}_{t_0}(z) = g_{t_0}(z) \sum_{k=1}^m \lambda_k(t_0)\cdot 
		\Phi(\xi_k(t_0),g_{t_0}(z);D_{t_0}),		
	\end{align}
	where $\lambda_1(t_0), \lambda_2(t_0)$ are uniquely determined non-negative numbers.\\
	If $t\mapsto h_{k;t}(z_0)$ is differentiable at $t_0$, then $t\mapsto h_{k;t}(z)$ is 
	differentiable at $t_0$ for every $z\in \Delta_k(t_0)$ and fulfills the following equation 
	\begin{align} \label{Equ:LE} \tag{$\star\star$}
		\dot{h}_{k;t_0}(z) = h_{k;t_0}(z)\cdot \mu_k(t_0)\cdot
		\frac{\zeta_k(t_0)+h_{k;t_0}(z)}{\zeta_k(t_0)-h_{k;t_0}(z)},
	\end{align}
	where $\mu_k(t_0)= \frac{d}{dt} \lmr(h_{k;t})|_{t=t_0}\ge0$. \\
	Moreover each function $t\mapsto \alpha_k(t)$ is continuous in $[0,T]$ for all $k=1,\ldots,m$
	and it holds $\alpha_k(t)>0$ and $\lambda_k(t_0) = \alpha_k^2(t_0) \cdot \mu_k(t_0)$.
\end{theorem}
\begin{remark}
The value $\lambda_k(t_0)$ can be given explicitly:\\
Let $t,\tau\in[0,T],$ set
\[
	\Omega_k(t,\tau):=\Omega\setminus \Big(\gamma_k[0,t] \cup \bigcup_{j=1\atop j\ne k}^m \gamma_j[0,\tau] 
		\Big)
\]		
and denote by $f_{k;t,\tau}$ the unique normalized conformal mapping from $\Omega_k(t,\tau)$ 
onto a circular slit disk. Then
\[
	\lambda_k(t_0)= \lim_{t\rightarrow t_0} \frac{\lmr(f_{k;t,t_0})-\lmr(f_{k;t_0,t_0})}{t-t_0},
	\]
	 see Lemma \ref{Lem-KLE-Lambda1}.
\end{remark}

\begin{remark}
Note that Theorem \ref{The-KLEvsLE} implies 
\begin{enumerate}
 \item[$\bullet$] Theorem \ref{The-KLEvsLE_Intro}: consider the case $\Omega=\D$,
 \item[$\bullet$] Theorem \ref{different_para}: let $m=1$,
 \item[$\bullet$] Theorem \ref{Simply_multiply} (disjoint case): apply  Theorem \ref{The-KLEvsLE} twice; first you pass from
 the multiply connected case with two slits to 
equation \eqref{Equ:LE}, then you pass to the simply connected case with two slits.
\end{enumerate}
Thus, what remains to show are Theorem \ref{counterexample}, Theorem \ref{lines}, Theorem \ref{Simply_multiply} for the branch point case 
 and Theorem \ref{The-KLEvsLE}.
\end{remark}
The rest of this paper is organized as follows: The proof of Theorem \ref{The-KLEvsLE} is given in Section 
\ref{sec_3} and in Section \ref{Cha-LocalBehaviour} we prove Theorem \ref{counterexample} and Theorem \ref{lines}. 
The proof of Theorem \ref{Simply_multiply} for the branch point case is given in the appendix.\\
We start with Section \ref{sec_2}, where we give three applications of Theorem \ref {The-KLEvsLE}.

	\section{Applications and examples}

\label{sec_2}
Theorem \ref{The-KLEvsLE} can be used to prove several results concerning
the Loewner equation for multiple slits. 
In this chapter we use the same notation as in Section \ref{organization} and we let $m=2$. \\

If we have no further information about the parametrizations $\gamma_k$ of the slits $\Gamma_k$ ($k=1,2$),
it is still possible to show that equation \eqref{Equ:KLE} holds for almost all $t\in[0,T].$\\
First, as the functions $t\mapsto \lmr(h_{k;t})$ are strictly increasing, the derivatives $\mu_1(t), \mu_2(t)$ exist almost everywhere. Thus we immediately get from 
Theorem \ref{Loewner_eq} that the functions $t\mapsto h_{k;t}(z)$ are differentiable almost 
everywhere for all $z\in \Delta_k(T)$ and all $k=1,2$. Together with 
Theorem \ref{The-KLEvsLE} we find the following Corollary, which has been already proved 
in \cite{BoehmLauf} by using different tools.
\begin{corollary}[Corollary 5 in \cite{BoehmLauf}]
	There exists a null-set $\mathcal{N}$ with respect to the Lebesgue measure such that the 
	functions $t\mapsto g_t(z)$ are differentiable on $[0,T]\setminus \mathcal{N}$ 
	for all $z\in \Omega_T$ and it holds
	\[
		\dot g_{t}{(z)} = g_{t}(z) \sum_{k=1}^m \lambda_k(t) \cdot \Phi(\xi_k(t),
		g_{t}(z);t)
	\]
	for all $t\in[0,T]\setminus\mathcal{N}$ and each $z\in\Omega_{t}$.
	Furthermore, the functions $\lambda_k(t_0)$ fulfill the condition 
	$\sum_{k=1}^m \lambda_k(t_0) = 1$ if the condition $g_{t}'(0)=c\,e^t$ holds in 
	a neighborhood of $t_0$ with some constant $c>0$.
\end{corollary}
Note that this is true for arbitrary parametrizations of the slits $\gamma_k$, i.e. we do not 
assume any normalization like $g_t'(0)=e^t$.\\

Next we will demonstrate how Theorem \ref{The-KLEvsLE} can be used to find new parametrizations 
for $\Gamma_1, \Gamma_2$, in order to get ``nice'' (Komatu-)Loewner equations,
i.e. equations with differentiability everywhere (and not only almost everywhere).\\
First, we let $L:=\lmr(g_T)$ and $L_k:=\lmr(h_{k;T}).$ Note that $L_k<L$ 
by the monotonicity of $\lmr$.
\begin{corollary} \label{Cor-Disjoint1}
    Assume $\gamma_1(0)\not=\gamma_2(0).$ Then there exist parametrizations $\tilde \gamma_1, \tilde \gamma_2\colon[0,L]\to\overline{\D}$ of the slits $\Gamma_1$ and $\Gamma_2$ such that the following
    holds: Denote by $\tilde g_s$ the unique normalized conformal mapping from
    $\tilde\Omega_s:=\Omega\setminus(\tilde \gamma_1[0,s]\cup \tilde \gamma_2[0,s])$
    onto a circular slit disk $\tilde D_s$ and let
    $\tilde \xi_k(s):=\tilde g_s(\tilde \gamma_k(s))$.\\
    Then the function  $s\mapsto \tilde g_s$ is continuously differentiable in $[0,L]$ with
    \begin{align} \label{Equ-KLE-everywhere}
        \dot{\tilde g}_s(z) = \tilde g_{s}(z) \sum_{k=1}^2
            \tilde \lambda_k(s) \cdot \Phi(\tilde \xi_k(s),\tilde g_s(z),\tilde D_s),\qquad
        \text{for all }s\in[0,L].
    \end{align}
    with continuous functions $\tilde \xi_k(s)$, $\tilde \lambda_k(s)\ge 0$ and
    $\tilde \lambda_1(s) +\tilde \lambda_2(s) = 1$ for all $s\in[0,L]$.
\end{corollary}
\begin{proof}
    First of all we  assume that each slit $\Gamma_k$ is parameterized in such a
    way that $\lmr( h_{k;t})$ is continuously differentiable for all $t\in[0,T]$, e.g.
    $\lmr( h_{k;t})=\frac{L_k}{T}\cdot t.$ (If not, then we can reparametrize
    $\gamma_1$ and $\gamma_2$.)\\[-0.7\baselineskip]

   In the notation of Theorem \ref{The-KLEvsLE}, this means that  $\mu_k(t)=\frac{L_k}{T}$ for all $t\in [0,T].$ 

    Then, by Theorem \ref{Loewner_eq} and Remark \ref{rm:continuous}, the trajectories $t\mapsto h_{k;t}(z)$ are continuously differentiable
    and fulfill equation (\ref{Equ:LE}) for each $t\in[0,T]$.

    By Theorem \ref{The-KLEvsLE}, the trajectories $t\mapsto g_t(z)$ fulfill equation (\ref{Equ:KLE}) for
    all $t\in[0,T]$. The right side of equation (\ref{Equ:KLE}) depends continuously on $t:$ \\
    the driving functions are continuous because of Remark \ref{rm:continuous}, and Lemma 19 in \cite{BoehmLauf} implies the continuity of the 
    function $\Phi.$ The continuity of the weights $t\mapsto\lambda_k(t)$ is an
    immediate consequence of the relation $\lambda_k(t)=\alpha_k^2(t)\cdot\mu_k(t)$ (see Theorem \ref{The-KLEvsLE}) together
    with the continuity of $\alpha_k$ and $\mu_k$.\\
    Hence, $t\mapsto g_t(z)$ is continuously differentiable in $[0,T]$.\\
    
    Note that, in general, the weights $t\mapsto\lambda_k(t)$ don't sum up to 1.\\[-0.7\baselineskip]

    In order to get normalized weights, we consider the following increasing homeomorphism
    $u(t):=\lmr(g_t)=\lmr(\gamma_1[0,t]\cup\gamma_2[0,t])$ that maps $[0,T]$ onto $[0,L]$.
    It is continuously differentiable with
    $\dot u(t)=\lambda_1(t)+\lambda_2(t)$ for all $t\in[0,T]$.
    This follows easily by differentiating \eqref{Equ:KLE} w.r.t. $z$ at the point $z=0$.

    Now we set $\tilde\gamma_k(s):=\gamma_k(u^{-1}(s))$ for all $s\in[0,L]$.
    Since the function $s\mapsto \lmr(\tilde\gamma_k[0,s])$
    is the composition of two continuously differentiable functions it is continuously
    differentiable as well.
    Consequently, by using Theorem \ref{The-KLEvsLE} in the same way as before,
    the trajectories $s\mapsto\tilde g_s(z)$
    are continuously differentiable and fulfill the stated differential equation for
    all $s\in[0,L]$.

    Finally, as $\lmr(\tilde g_s)=s$ for all $s\in[0,L]$, we have
    $ \tilde\lambda_1(s)+ \tilde\lambda_2(s)=1$ for all $s\in[0,L]$.
\end{proof} 
\begin{remark}
 The proof of Corollary \ref{Cor-Disjoint1} shows that there exist ``many'' parametrizations
 $\tilde \gamma_1,$ $\tilde \gamma_2$ such that equation \eqref{Equ-KLE-everywhere} holds and tells us
 how to construct them. This is based on the fact that we are not restricted to claim $\lmr(h_{k;t})=\frac{L_k}{T}\cdot t$. Instead, we 
 can choose the initial parametrization in such a way that $\lmr(h_{k;t})=u_k(t)$ holds, 
 where $u_k\colon[0,T]\rightarrow [0,L_k]$ is an arbitrary
  continuously differentiable increasing homeomorphism.\\
In \cite{BoehmSchl} it was shown that one can even choose $\tilde \gamma_1$ and $\tilde \gamma_2$ such that
$\tilde \lambda_1(t)$ and $\tilde \lambda_2(t)$ are constant. Furthermore, this additional
condition makes $\tilde \gamma_1$ and $\tilde \gamma_2$ unique.
\end{remark}

The next application is a bit more technical, but quite useful, e.g. for constructing certain
counterexamples mentioned in the introduction.\\

Assume that $u_1\colon[0,L]\to[0, L_1]$ is a given increasing homeomorphism. It is easy to see that 
we can find an increasing homeomorphism $v_1\colon[0,L]\to [0,T],$
such that $\lmr(h_{1,v_1(s)})=u_1(s)$ for all $s\in[0,L].$ Now consider the following question:\\
Can we find an increasing homeomorphism $v_2\colon[0,L]\to [0,T]$ such that the function
$s\mapsto \tilde g_s$ from Corollary \ref{Cor-Disjoint1} satisfies equation \eqref{Equ-KLE-everywhere}
with $\tilde \lambda_1(s)+\tilde \lambda_2(s)=1$ for all $s\in[0,L]$?\\

The following statement gives a partial answer to this question for the simply connected case,
i.e. $\Omega = \D.$ The proof depends on an inequality
for the logarithmic mapping radius (see inequality (\ref{Equ-lmrInequality}))
that is only known to be true for the simply connected case. 
\begin{proposition}\label{Cor-Disjoint2}
	Assume $\gamma_1(0)\not=\gamma_2(0)$. Let $\Omega=\D$ and $u_1\colon[0,L]\rightarrow [0,L_1]$ be an increasing Lipschitz continuous 
	function with a Lipschitz constant $K< 1$.\\
	Let $v_1\colon[0,L]\to [0,T]$ be the increasing homeomorphism such that 
	$\lmr(h_{1,v_1(s)})=u_1(s)$ for all $s\in[0,L].$ 
	Then there is a unique increasing homeomorphism $v_2\colon[0,L]\to [0,T]$ such that
	the following holds:\\
	Denote by $\tilde g_s$ the unique normalized conformal mapping from
	$\Omega_s:=\D\setminus((\gamma_1\circ v_1)[0,s]\cup (\gamma_2\circ v_2)[0,s])$
	onto $\D$. Then $\lmr(\tilde g_s)=s$ for all $s\in[0,L]$.\\
	Moreover, if $s\mapsto u_1(s)$ is continuously differentiable in $[0,L]$, then the function 
	$s\mapsto \tilde g_s(z)$ is continuously differentiable and satisfies equation 
	\eqref{Equ-KLE-everywhere} for all $s\in[0,L]$ and all
	$z\in\D\setminus(\Gamma_1\cup\Gamma_2)$ with 
	$\tilde \lambda_1(s)+\tilde \lambda_2(s)= 1$ for all $s\in[0,L]$.
\end{proposition}
The proof of this proposition is given in Section \ref{sec_3}.

\begin{example} \label{Exa-Counterexample2}
	Let $\Omega=\D$ and $\Gamma_1$, $\Gamma_2$ be disjoint slits with 
	$L=1.$ Then $L_1, L_2<1$.\\ 
	Consequently we find an $\epsilon >0$ so that $L_1+\epsilon < 1$ as well. Then we define
	\[
		u_1:[0,1] \rightarrow [0,L_1],\quad s\mapsto u_1(s):=
		\begin{cases}
			(L_1+\epsilon) s & \text{if }t\in[0,\frac{1}{2}],\\
			(L_1-\epsilon) s + \epsilon & \text{if }s\in(\frac{1}{2},1].
		\end{cases}
	\]
	By $v_1:[0,1]\rightarrow [0, L_1]$ we denote the homeomorphism such that
	$\lmr(h_{1;v_1(s)})=u_1(s)$. \\
	$u_1$ is Lipschitz continuous with Lipschitz constant $K=L_1+\epsilon<1$. By Proposition \ref{Cor-Disjoint2}
	we find a homeomorphism $v_2\colon[0,1]\rightarrow [0,L_2]$ so that $\lmr(\tilde g_s)=s$ for all $s\in[0,1]$.\\
	The function $s\mapsto h_{1;v_1(s)}$ is not differentiable at $s=\frac{1}{2}$
	by Theorem \ref{Loewner_eq} as $u_1(s)=\lmr(h_{1,v_1(s)})$ is not differentiable at $s=\frac{1}{2}$.\\
	Thus, by Theorem \ref{The-KLEvsLE}, the function $s\mapsto \tilde g_s$ is not differentiable at $s=\frac{1}{2}$. 
	However, the function $s\mapsto \lmr(\tilde g_s)=s$ is differentiable at $s=\frac{1}{2}$. \hfill $\bigstar$
\end{example}

Finally, we consider the slightly different setting of two slits with one common starting point.
The next example shows that the converse of Theorem \ref{lines} is not true.

\begin{example} \label{Exa-Counterexample3}
	Let $\gamma_1$, $\gamma_2\colon[0,T]\to\overline\D$ be parametrizations of two slits satisfying
	the conditions of Theorem \ref{lines}. Let $g_t$ be defined as in Theorem \ref{lines}. \\
	Furthermore, let $h_{k;t}$ be the unique normalized mapping from $\D\setminus \gamma_k[0,t]$
	onto $\D.$ \\
	Without restricting generality we may assume 
	$L:=\lmr(g_T)=1$. Moreover, let $L_k:=\lmr(h_{k;T})$. Then $L_k<1$ and
	we find analogously to Example \ref{Exa-Counterexample2} an
	$\epsilon>0$ so that $L_1+\epsilon<1$.
	
	Next, let $u\colon[0,1]\rightarrow[0,L_1]$ be defined by
		\[
		s\mapsto u(s)=
		\begin{cases}
			(L_1+\epsilon) s & \text{if }t\in[0,\frac{1}{2}],\\
			(L_1-\epsilon) s + \epsilon & \text{if }s\in(\frac{1}{2},1].
		\end{cases}
	\]
	We will use $u$ to construct another increasing homeomorphism
	$u_1\colon[0,1]\rightarrow [0,L_1]$ (see Figure \ref{es_regnet}):
	\[ 
		u_1(s):={\begin{cases}
				\frac{1}{2^n} u(2^ns-1) + \frac{L_1}{2^n} & \text{ if }s\in(\frac{1}{2^n},\frac{2}{2^n}] \text{ with } n\in\N,\\
				\mbox{}\hspace{1cm}0& \text{ if }s=0.
			\end{cases}}
	\]
	We have $|u_1(t_2)-u_1(t_1)|\le(L_1+\epsilon)(t_2-t_1)$ for all $0\le t_1\le t_2\le 1$,
	so $u_1$ is strictly increasing and Lipschitz continuous. Moreover we denote by $v_1\colon[0,1]\to [0,T]$ the unique
	homeomorphism having the property that $\lmr(h_{1;v_1(s)})=u_1(s)$ holds for all $s\in[0,1]$.
	Now we find a unique homeomorphism $v_2\colon[0,1]\to [0,T]$ such that 
	$\lmr(\tilde g_s)=s$ holds for all $s\in[0,1]$, where $\tilde g_s$ denotes the unique normalized
	mapping from $\D\setminus((\gamma_1\circ v_1)[0,s]\cup (\gamma_2\circ v_2)[0,s])$ onto $\D$. This is possible to do using the first part of the proof of Proposition \ref{Cor-Disjoint2}, as it is applicable to the branch point case as well. \\
	On the one hand, by Theorem \ref{Simply_multiply}, the function $s\mapsto \tilde g_s$ is differentiable at $s=0$.
	On the other hand, $s\mapsto h_{1;v_1(s)}$ is not differentiable at $s=0$ in accordance with Theorem \ref{Loewner_eq}, because by construction,
	 $\lmr(h_{1;v_1(s)})=u_1(s)$ and $u_1'(0)$ does not exist. \hfill $\bigstar$
\end{example}

\begin{figure}[h]\label{es_regnet}
\centering
	\begin{overpic}[scale=0.5, bb=75 213 542 585]
		{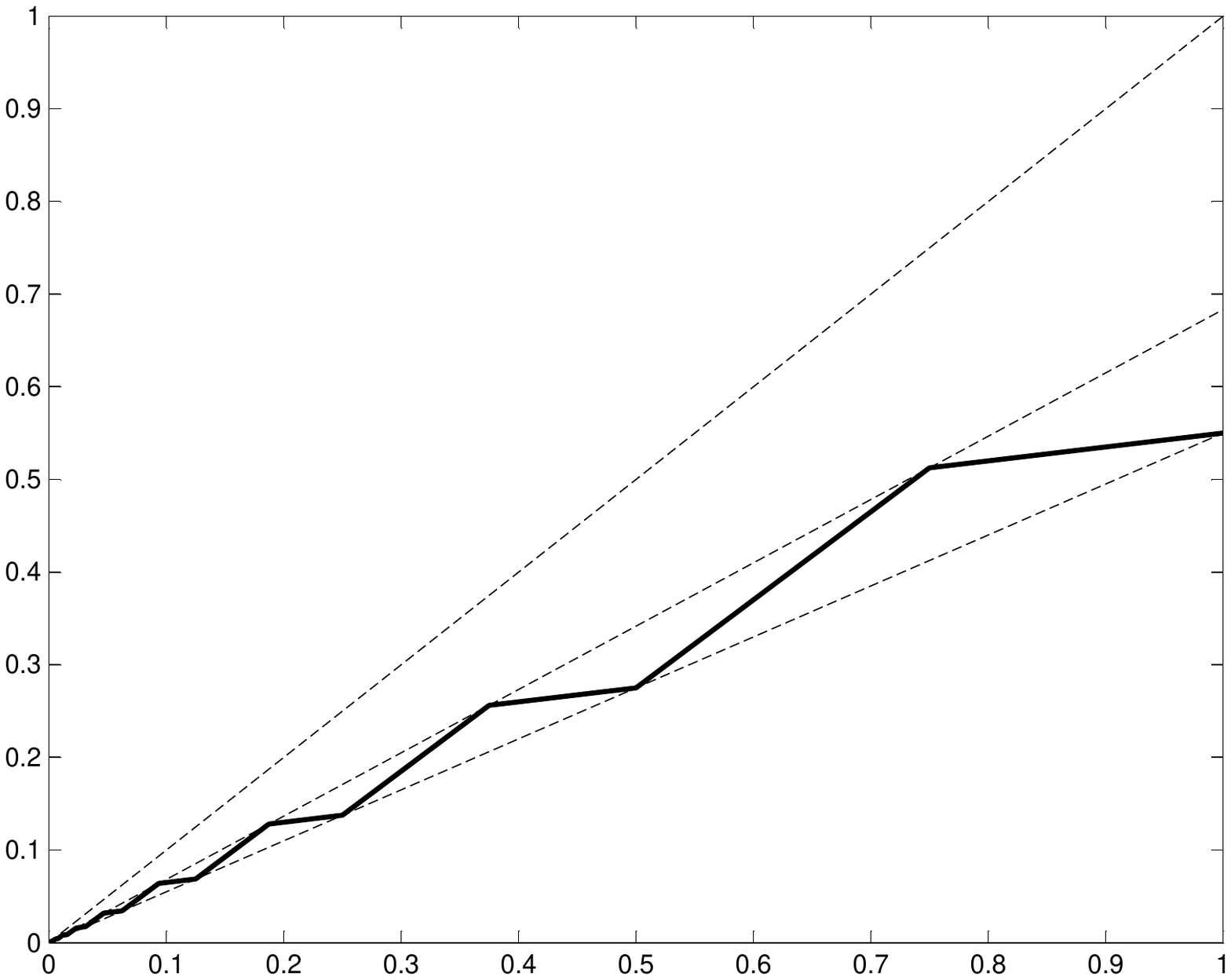}
	\end{overpic}
	\caption{The function $u_1$ from Example \ref{Exa-Counterexample3}.}
\end{figure}

	\section{Proof of Theorem \ref{The-KLEvsLE} and Proposition \ref{Cor-Disjoint2}} 
	\label{sec_3}
As we have mentioned in the introduction, all statements can be easily generalized to the case
$m>2$, so we will use a notation indicating this case as well.

First of all, for all $t,\tau\in[0,T],$ we set 
\[
	\Omega_k(t,\tau):=\Omega\setminus \Big(\gamma_k[0,t] \cup \bigcup_{j=1\atop j\ne k}^m \gamma_j[0,\tau] 
	 \Big)
\]		
and denote by $f_{k;t,\tau}$ the unique normalized  mapping $f_t: \Omega_k(t,\tau) \to D_k(t,\tau),$ 
where $D_k(t,\tau)$ is a circular slit disk. 
Consequently we have $g_t=f_{k;t,t}$ and $\Omega_t=\Omega_{k}(t,t)$ as well.\\[\lenpar]
Next, provided that the limits exist, we define
\[
	\lambda_k(t_0):= \lim_{t\rightarrow t_0} \frac{\lmr(f_{k;t,t_0})-\lmr(f_{k;t_0,t_0})}{t-t_0},
	\qquad
	\mu_k(t_0):= \lim_{t\rightarrow t_0} \frac{\lmr(h_{k;t})-\lmr(h_{k;t_0})}{t-t_0}.
\]
Finally we set $\xi_k(t,\tau):=f_{k;t,\tau}\big(\gamma_k(t)\big)$,
\[
	S_{k;\lt,\gt,\tau}:= f_{k;\lt,\tau}\big(\gamma_k[\lt,\gt] \big)
		\subset \D\cup \{\xi_k(t,\tau)\}, \quad
	s_{k;\lt,\gt,\tau}:=f_{k;\gt,\tau}\big(\gamma_k[\lt,\gt] \big)
		\subset \partial\D,
\]
and $\sigma_{k;\lt,\gt}:=h_{k;\gt}(\gamma_k[\lt,\gt])$ for all $0\le \lt\le\gt\le T$.
Next we are going to use results from \cite{BoehmLauf} in order to show that the existence
of the above limits $\lambda_k(t_0)$ is equivalent to differentiability
of the function $t\mapsto g_t(z)$.
\begin{lemma} \label{Lem-KLE-Lambda1}
	Let $t_0\in[0,T]$. 
	Then the following three conditions are equivalent:
	\begin{enumerate}
		\item Each limit $\lambda_k(t_0)$ exists ($k=1,\ldots,m$).
		\item The function $t\mapsto g_t(z)$ is differentiable at $t_0$ for every 
			$z\in\Omega_{t_0}$.
		\item The function $t\mapsto g_t(z)$ is differentiable at $t_0$ for every 
			$z\in\Omega_{t_0}$ and fulfills equation (\ref{Equ:KLE}) 
			for all $z\in\Omega_{t_0}$.
	\end{enumerate}	
\end{lemma}
\begin{proof}
	First of all note that (1.)$\Rightarrow$(3.) follows immediately from Theorem 2 of 
	\cite{BoehmLauf}. On top of this, (3.)$\Rightarrow$(2.) is trivial, so the only thing
	we are going to prove is (2.)$\Rightarrow$(1.).
	
	For this, let $t_0\in[0,T]$ and $t>t_0$. The other case $t<t_0$ can be treated
	in the same way.
	Since $t\mapsto g_t(z)$ is differentiable at $t_0$ so is $t\mapsto \log(g_t(z))$ 
	for every $z\in\Omega_{t_0}\setminus\{0\}$. The function $\log(g_t(z))$ is
	multiple valued, but its derivative is single valued, so the following limit exists and is 
	independent of the branch of the logarithm:
	\[
		\lim_{t\searrow t_0}\frac{1}{t-t_0}\,\Re\left(\log\frac{g_t(z)}{g_{t_0}(z)}\right) = 
		\lim_{t\searrow t_0}\frac{1}{t-t_0}\, \ln\left|
			\frac{g_t(z)}{g_{t_0}(z)}\right|.
	\]
	Next we use Lemma 10 from \cite{BoehmLauf} and the mean value theorem to get 
	\begin{align}\label{Equ-Boundedness}\begin{split}
		\ln\left|\frac{g_t(z)}{g_{t_0}(z)}\right| &=
		\frac{1}{2\pi} \sum_{k=1}^m \int_{s_{k;t_0,t,t}} -\ln|(g_{t_0}\circ g_t^{-1})(\xi)|\;
			\Re\big(\Phi(\xi, g_t(z); t)\big)\, |\intd\xi|\\
		&= \frac{1}{2\pi} \sum_{k=1}^m \Re\big(\Phi(\xi^{(k)}_{t,t_0}, g_t(z); t)\big) 
			\int_{s_{k;t_0,t,t}} -\ln|(g_{t_0}\circ g_t^{-1})(\xi)| \, |\intd\xi|.
	\end{split}\end{align}
	Note that $\Phi(\xi^{(k)}_{t,t_0},g_t(z);t)$ tends to $\Phi(\xi_k(t_0),g_{t_0}(z);t_0)$ as $t\searrow t_0$
	by Lemma 19 from \cite{BoehmLauf}. This is based on the fact
	$s_{k;t_0,t,t}\ni\xi^{(k)}_{t,t_0}\rightarrow \xi_k(t_0)$ as $t\searrow t_0$, see 
	Proposition 8 from \cite{BoehmLauf}. We write
	\[
		\int_{s_{k;t_0,t,t}} -\ln|(g_{t_0}\circ g_t^{-1})(\xi)| \, |\intd\xi|
		=:c_k(t,t_0).
	\]
	Note that for each $k\in\{1,\ldots,m\}$, $\Re\big(\Phi(\xi^{(k)}_{t,t_0}, g_t(z); t)\big)$ and $c_k(t,t_0)$ are 
	positive. Moreover, the limit $\lim_{t\searrow t_0} \frac{1}{t-t_0} \ln|g_t(z)/g_{t_0}(z)|$
	exists by assumption for any $z\in\Omega_{t_0}$. Summarizing, equation \eqref{Equ-Boundedness} shows that 
	$\frac{c_k(t,t_0)}{t-t_0}$ is bounded for all $t\in(t_0,T]$. Together with Lemma 19 from \cite{BoehmLauf} we find
	\begin{align}\label{Equ-LinearCombination}
		\ln\left|\frac{g_t(z)}{g_{t_0}(z)}\right|= \frac{1}{2\pi} \sum_{k=1}^m
		 \Re\big(\Phi(\xi_k(t_0), g_{t_0}(z); t_0)\big) 
			c_k(t,t_0) + o(|t-t_0|).
	\end{align}
	From the proof of Theorem 2 of \cite{BoehmLauf} we can see that 
	$\lambda_k^+(t_0):=\lim_{t\searrow t_0}\frac{\lmr(f_{k;t,t_0})-\lmr(f_{k;t_0,t_0})}{t-t_0}$ 
	exists if and only if $\lim_{t\searrow t_0}\frac{c_k(t,t_0)}{t-t_0}$ exists.
	Consequently we are going to prove the existence of the limit 
	$\lim_{t\searrow t_0}\frac{c_k(t,t_0)}{t-t_0}$.
	
	For this purpose we show that we find $z_1,\ldots,z_m\in\Omega_{t_0}$ (independently of $t$) such that 
	the matrix $A:=[a_{j,k}]_{j,k=1}^m$, with $a_{j,k}:=\Re\big(\Phi(\xi_k(t_0), g_{t_0}(z_j); t_0)\big)$, 
	is invertible. Then, equation \eqref{Equ-LinearCombination} yields
	\[
		\big(c_1(t,t_0),\ldots,c_m(t,t_0)\big)^T = \frac{1}{2\pi} A^{-1} \big(\ln|\tfrac{g_t(z)}{g_{t_0}(z_1)}| ,\ldots  \ln|\tfrac{g_t(z)}{g_{t_0}(z_m)}|\big)^T + o(|t-t_0|),
	\]
	and the existence of the limits $\lim_{t\searrow t_0}\frac{c_k(t,t_0)}{t-t_0}$ follows immediately.
	
	To find $z_1,\ldots,z_m,$ recall that
	$\Phi(\xi_k(t_0), g_{t_0}(\gamma_k(t_0));t_0) = \infty$ and
	$\Re\big(\Phi(\xi_k(t_0), g_{t_0}(\gamma_j(t_0));t_0)\big) = 0$ if $j\ne k$. 
	For $k\in\{1,\ldots,m\}$ consider the preimage $L_k$ of the curve 
	$\delta(x):=1+\mathrm{i} x$, $x\ge 0$, under the mapping $z\mapsto \Phi(\xi_k(t_0),g_{t_0}(z);t_0)$.
	Since $L_k$ is a slit in $\Omega(t_0)$ landing at the point $\gamma_k(t_0)$, 
	$\Re(\Phi(\xi_j(t_0),g_{t_0}(z);t_0))\rightarrow 0$ as $z\in L_k$ tends to 
	$\partial\Omega(t_0)$ when $j\ne k$, while $\Re(\Phi(\xi_k(t_0),g_{t_0}(z);t_0))=1$ for 
	all $z\in L_k$ by construction.
	
	Thus we can choose $z_k\in\Omega(t_0)$ close enough to $\gamma_k(t_0)$ in order to get
	\[
		\Re\big(\Phi(\xi_k(t_0), g_{t_0}(z_k); t_0)\big) = 1,\qquad
		\Re\big(\Phi(\xi_j(t_0), g_{t_0}(z_k); t_0)\big) < \tfrac{1}{m} \text{ for all }j\ne k.
	\]
	Consequently the matrix $A$ is a diagonally dominant matrix, so 
	it is invertible as well.
	\begin{center}
	\begin{overpic}[scale=\scalefactor]
		{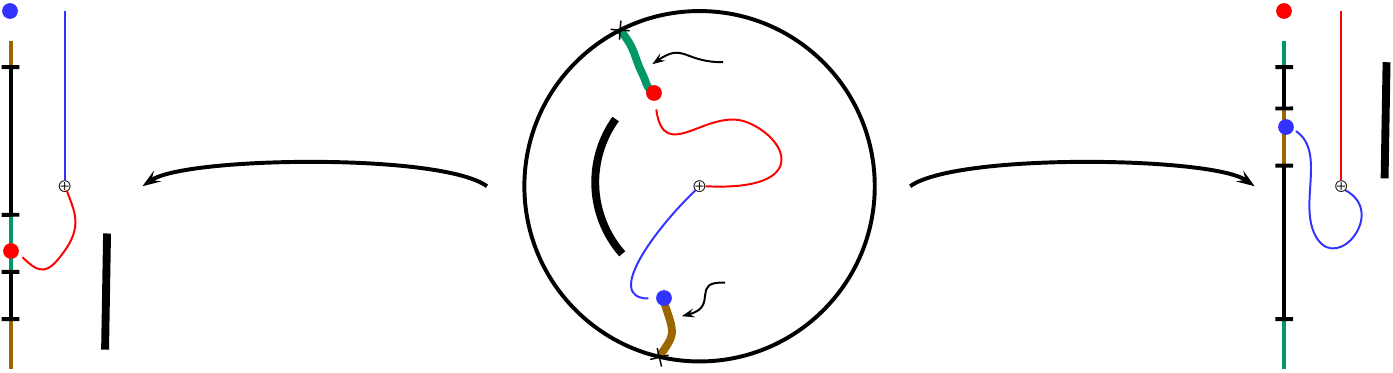}
		\put(52,22){$j$}
		\put(52,5){$k$}
		\put(10,16.5){$\Phi(\xi_k(t_0),g_{t_0}(z),t_0)$}
		\put(65,16.5){$\Phi(\xi_l(t_0),g_{t_0}(z),t_0)$}
	\end{overpic}
	\end{center}

\end{proof}

The next lemma is a similar statement for the functions $h_{k;t}.$ Note, however, that here we only
need differentiability of $t\mapsto h_{k;t}(z_0)$ for one fixed $z_0\in \Delta_k(t_0)\setminus\{0\}.$

\begin{lemma} \label{Lem-KLE-Lambda2}
	Let be $t_0\in[0,T]$, $z_0\in\Delta_k(t_0)\setminus\{0\}$ and $k\in\{1,\ldots,m\}$. 
	Then the following three conditions are equivalent
	\begin{enumerate}
		\item The limit $\mu_k(t_0)$ exists.
		\item The function $t\mapsto h_{k;t}(z_0)$ is differentiable at $t_0$.
		\item The function $t\mapsto h_{k;t}(z)$ is differentiable at $t_0$ for every 
			$z\in\Delta_k(t_0)$ and and fulfills equation (\ref{Equ:LE})
			for all $z\in\Delta_k(t)$.
	\end{enumerate}	
\end{lemma}
\begin{proof}
	First of all note that (1.)$\Rightarrow$(3.) follows immediately from Theorem 2 from
	\cite{BoehmLauf}. On top of this, (3.)$\Rightarrow$(2.) is trivial, so the only thing
	we need to prove is (2.)$\Rightarrow$(1.).
	
	Let $t>t_0$ and and $k\in\{1,\ldots,m\}$.
	Analogous to the proof of the previous lemma, we find
	\begin{align*}
		\log\left(\frac{h_{k;t}(z_0)}{h_{k;t_0}(z_0)} \right) &= 
		\frac{1}{2\pi} \int_{\sigma_{k;t,t_0}} 
			-\ln|(h_{k;t_0}\circ h_{k;t}^{-1})(\zeta)|\;
			\frac{\zeta_k(t)+h_{k;t}(z_0)}{\zeta_k(t)-h_{k;t}(z_0)}\, |\intd\zeta|\\
		&= \frac{\zeta_k(t_0)+h_{k;t_0}(z_0)}{\zeta_k(t_0)-h_{k;t_0}(z_0)}\,
			\frac{1}{2\pi} \int_{\sigma_{k;t,t_0}} 
			-\ln|(h_{k;t_0}\circ h_{k;t}^{-1})(\zeta)|\;|\intd\zeta| + o(|t-t_0|).
	\end{align*}
	The other case $t<t_0$ holds in the same way.
	From the proof of Theorem 2 of \cite{BoehmLauf} we can see that the limit
	\[
		\lim_{t\searrow t_0} \frac{1}{t-t_0} \int_{\sigma_{k;t,t_0}} 
			-\ln|(h_{k;t_0}\circ h_{k;t}^{-1})(\zeta)|\;|\intd\zeta|
	\]
	exists if and only if $\mu_k^+(t_0)$ exists, where 
	\[
		\mu^+_k(t_0):=\lim_{t\searrow t_0} \frac{\lmr(h_{k;t})-\lmr(h_{k;t_0})}{t-t_0}.
	\]
	Since $t\mapsto h_{k;t}(z_0)$ is differentiable at $t_0$ the proof is complete
\end{proof}

\begin{remark}
 The implication (2.)$\Rightarrow$(3.) in the previous lemma says that differentiability of 
 $t\mapsto h_{k;t}(z)$ at $t_0$ for just
 one point $z_0\in \Delta_k(t_0)\setminus\{0\}$ implies differentiability at $t_0$ for all $z\in \Delta_k(t_0).$\\
 We don't know whether the same is true in the case of $m>1$ slits. Note, however, that
 the proof of Lemma \ref{Lem-KLE-Lambda1} shows that there are $m$ points such that differentiability
 of $t\mapsto g_t(z_1),... , t\mapsto g_t(z_m)$ at $t_0$ together implies differentiability
 of $t\mapsto g_t(z)$ for all $z\in \Omega_{t_0}.$
 \end{remark}

Before we can proof Theorem \ref{The-KLEvsLE}, we need some preliminary lemmas.
\begin{lemma} \label{Lem-Inequality}
	Let $A, B \subset \D$ be bounded domains and  assume there exists 
	an $R>0$ so that
	\[
		A\cap B_{R}(1) = B\cap B_{R}(1) = \D\cap B_{R}(1)
	\]
	holds, where $B_{R}(z_0):=\{z\in\C\mid |z-z_0|<R\}$.
	Moreover let $T:A\rightarrow B$ be a conformal mapping from $A$ onto $B$, 
	where $T(1)=1$.\\[\lenpar]
	Then $c:=T'(1)>0$ and for any $\delta>0$ there exists $\epsilon >0$ such that
	the inequality
	\[
		|z|^{c+\delta} \le |T(z)| \le |z|^{c-\delta}
	\]
	holds for all $z\in A\cap B_{\epsilon}(1)$.
\end{lemma}
\begin{proof}
First of all, we can extend the function $T$ to a conformal map in $B_{R}(1)$ by using the Schwarz reflection principle. 
As the arc $\partial\D\cap B_R(1)$ is mapped onto an arc of $\partial \D$ and $T(1)=1$,
we have $c:=T'(1)>0$.
\begin{center}
\begin{overpic}[scale=\scalefactor]
	{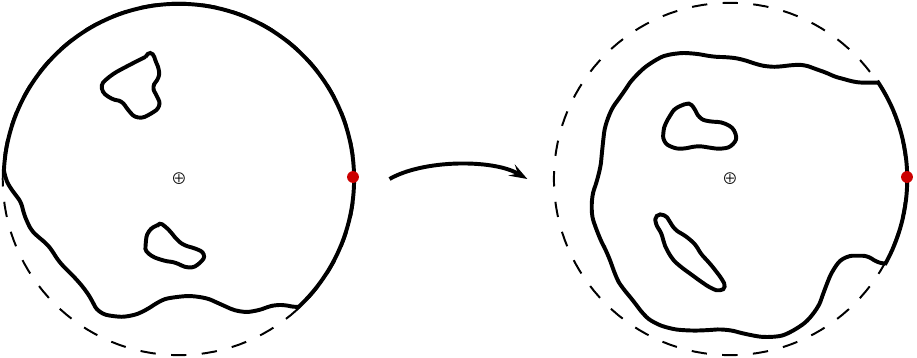}
	\put(25,25){$A$}
	\put(87,20){$B$}
	\put(35.5,19){$1$}
	\put(96,19){$1$}
	\put(20.5,19){$0$}
	\put(80.5,19){$0$}
	\put(49.5,22){$T$}
\end{overpic}
\end{center}
Now we can choose $\epsilon\in(0,R)$ small enough such that 
\[
	\Big| \frac{T'(z)}{T(z)} - \frac{c}{z}\Big| <\delta,\quad \text{for all }z\in B_{\epsilon}(1).
\]
Next we set $\gamma_\theta(r):=r\cdot e^{i \theta}$ for all $r\in[r_0,1]$ and all 
$|\theta|<\phi$. Hereby we can choose $r_0$ close enough to $1$ and $\phi>0$ small enough
to get $\gamma_\theta(r)\in B_\epsilon(1)$ for all $r\in[r_0,1]$ and all $\theta\in(-\phi,\phi)$.
Moreover, for $r\in[r_0,1]$ and $\theta\in(-\phi,\phi),$ we define
\[
	h_\theta(r):=\Re\left( \log \frac{T(\gamma_\theta(r))}{(\gamma_\theta(r))^c} \right)
		= \ln \left|\frac{T(\gamma_\theta(r))}{(\gamma_\theta(r))^c} \right|.
\]
Note that there is an analytic branch of the logarithm of $\frac{T(z)}{z^c}$ in $B_\epsilon(1)$, 
so we find
\begin{align*}
	\left| \frac{\partial}{\partial r}{h}_\theta(r)\right| &= 
	\left|\Re\left(\left. {\frac{\intd}{\intd z} 
		\log\left( \frac{T(z)}{z^c} \right)}\right|_{z=\gamma_\theta(r)}
	\cdot \dot\gamma_\theta(r) \right)\right|\\
	&= \left|\Re\left(\left.\left(\frac{T'(z)}{T(z)} - \frac{c}{z}\right)
		\right|_{z=\gamma_\theta(r)} \cdot e^{i \theta} \right)\right|\\
	&\le \left|\left.  \frac{T'(z)}{T(z)} - \frac{c}{z} \right|_{z=\gamma_\theta(r)}
	\right| \le \delta.
\end{align*}
Moreover, we have $h_{\theta}(1)=0$ so we find
\[
	\ln(r^\delta) =\delta \ln(r) \le h_\theta(r) \le -\delta \ln(r) = \ln(r^{-\delta}).
\]
Finally we get $\ln(|z|^\delta)\le |\frac{T(z)}{z^c}| \le \ln(|z|^{-\delta})$ for all
$z\in \{r\cdot e^{i \theta}\mid r\in[r_0,1],\, \theta\in(-\phi,\phi)\}$, so the proof is 
complete.
\end{proof}
\begin{lemma} \label{Lem-ContinuityAlpha}
    The function $t\mapsto \alpha_k(t)$ is continuous and positive for all $t\in[0,T]$.\\
   Moreover, for $t_0\in[0,T],$ 
	\[
		\Big|\big(f_{k;t,\tau}\circ h_{k;t}^{-1}\big)'\big(a\big) \Big|
		\rightarrow \alpha_k(t_0)
	\]
	as $[0,T]^2\times \partial \D \ni (t,\tau, a) \rightarrow (t_0,t_0, \zeta_k(t_0))$.
\end{lemma}
\begin{proof}
	First of all, $\alpha_k(t)$ is positive, as the mapping $g_{t}\circ h_{k;t}^{-1}$ can be
	extended analytically to a conformal map in a small neighborhood around $\zeta(t)$.
	Consequently the derivative can not vanish. 
	
	The continuity of $\alpha_k$ follows from the second statement of the lemma, which we are going to prove below, because 
	\[
		\Big|\big(f_{k;t,t}\circ h_{k;t}^{-1}\big)'\big(\zeta_k(t)\big)\Big| = \alpha_k(t)
	\]
	holds for all $t\in[0,T]$ and because $\partial\D\ni\zeta_k(t)\to\zeta_k(t_0)$ as $t\rightarrow t_0$ by Remark \ref{rm:continuous}.
	
	Note that we find an $\epsilon >0$, so that the mapping 
	$H_{k;t,\tau}:=f_{k;t,\tau}\circ h_{k;t}^{-1}$ extends analytically to $B_{\epsilon}(\zeta_k(t))$ 
	by the Schwarz reflection principle.
	Since $\zeta_k(t)\rightarrow \zeta_k(t_0)$ as $t$ tends to $t_0$, 
	we find a small neighborhood $U$ around $\zeta_k(t_0)$, where $H_{k;t,\tau}$ is analytic if
	$t$ and $\tau$ are close enough to $t_0$. By Proposition 7 from \cite{BoehmLauf}, 
	$H_{k;t,\tau}$ converges locally uniformly in $U\cap\D$ to $H_{k;t_0,t_0}$ as $(t,\tau)\to(t_0,t_0)$. 
    Using a normality argument, it is easy to see that $H_{k;t,\tau}$ converges in fact locally uniformly on $U$ to $H_{k;t_0,t_0}$, so we have
	\[
		H_{k;t,\tau}\big(a\big) \rightarrow H_{k;t_0,t_0}\big(\zeta_k(t_0)\big)
		\quad \text{as}\quad [0,T]^2\times \partial \D \ni (t,\tau, a)\rightarrow (t_0,t_0, \zeta_k(t_0)).
	\]
	Finally, we	find $\big|H'_{k;t,\tau}\big(a\big)\big| \rightarrow 
	\big|H'_{k;t_0,t_0}\big(\zeta_k(t_0)\big)\big| = \alpha_k(t_0)$ as  $(t,\tau, a)\rightarrow (t_0,t_0, \zeta_k(t_0))$,
	so the proof is complete.
\end{proof}
\begin{lemma} \label{Lem-ConnectionMuLambda}
	Let be $t_0\in[0,T]$ and $k\in\{1,\ldots,m\}$. Then 
	\[
		\lim_{t\rightarrow t_0} \frac{\lmr(f_{k;t,t_0})-\lmr(f_{k;t_0,t_0})}
			{\lmr(h_{k;t})-\lmr(h_{k;t_0})} = \alpha_k^2(t_0).
	\]
	Consequently the limit $\lambda_k(t_0)$ exists if and only if the limit $\mu_k(t_0)$ exists. 
	Moreover, in this case $\lambda_k(t_0) = \alpha_k^2(t_0) \cdot \mu_k(t_0)$ holds.
\end{lemma}
\begin{proof}
	First of all we are going to prove the case $t\searrow t_0$, i.e. we show that
	\[
		\lim_{t\searrow t_0} \frac{\lmr(f_{k;t,t_0})-\lmr(f_{k,t_0,t_0})}{t-t_0}
		= \alpha^2_k(t_0)\cdot \lim_{t\searrow t_0} \frac{\lmr(h_{k;t})-\lmr(h_{k;t_0})}{t-t_0}.
	\]
	Let be $t_0\in[0,T]$ and $k\in\{1,\ldots,m\}$. Since there is no risk of confusion we
	omit the index $k$.
	\begin{center}
	\begin{overpic}[scale=\scalefactor]
		{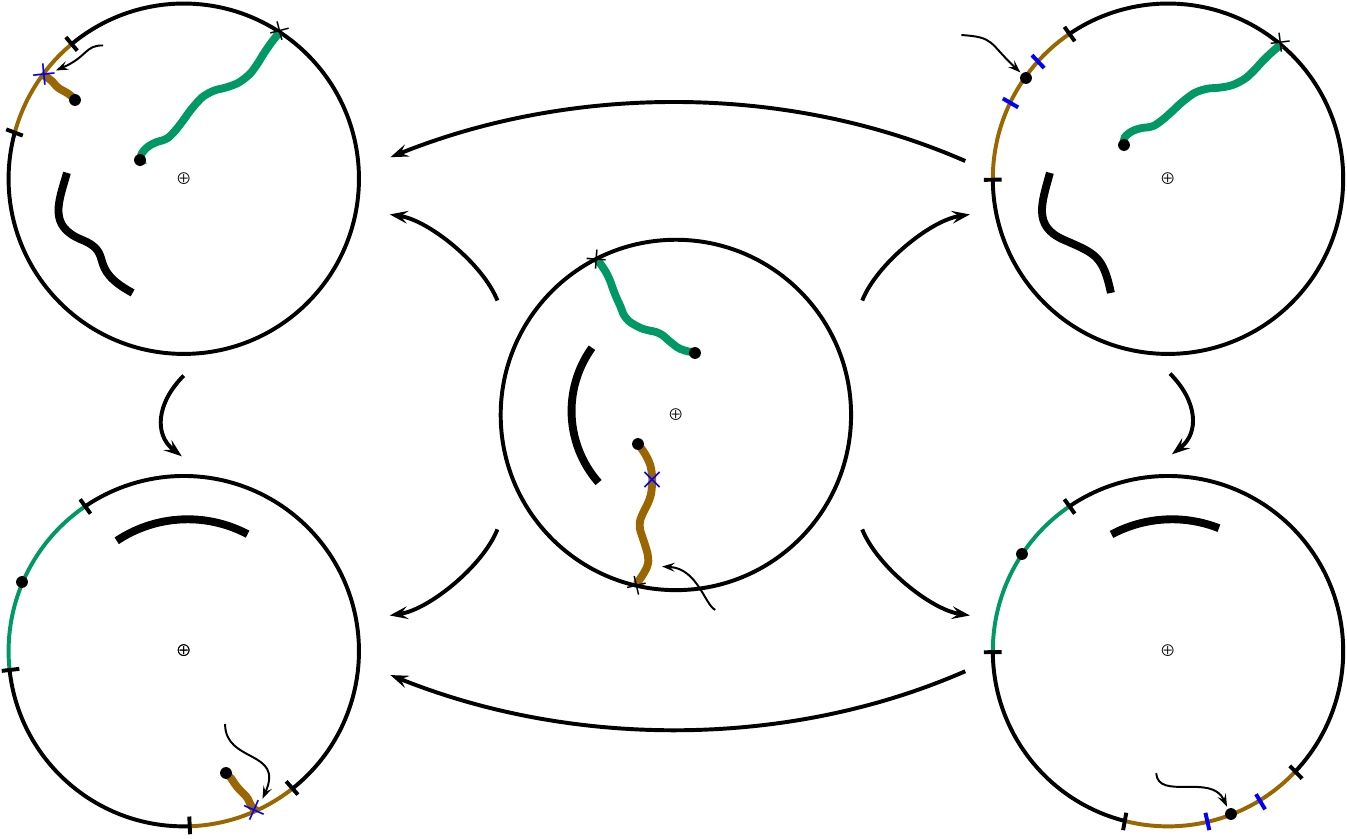}
		\put(67,19){$f_{t,t_0}$}
		\put(67,41){$h_{t}$}
		\put(29.5,22){$f_{t_0,t_0}$}
		\put(31,42){$h_{t_0}$}
		\put(53,15){$k$}
		\put(51,37){$t_0$}
		\put(49,25){$t_0$}
		\put(46,30){$t$}
		\put(47,5){$G_{t,t_0}$}
		\put(47,56.3){$F_{t,t_0}$}
		\put(7,30){$T_{t_0}$}
		\put(89,30){$H_{t,t_0}$}
		\put(84,17){$D(t,t_0)$}
		\put(89,-2){\color{blue}$s_{t_0,t,t_0}$}
		\put(76.5,54){\color{blue}$\sigma_{t_0,t}$}
		\put(15,9){$\xi(t_0)$}
		\put(82,6){$\xi(t,t_0)$}
		\put(8,58){$\zeta(t_0)$}
		\put(66,59){$\zeta(t)$}
	\end{overpic}
	\end{center}
	 Then we have with $G_{t,t_0}:= f_{t_0,t_0}\circ f_{t,t_0}^{-1},$
	\begin{align*}
		&\lmr(f_{t_0,t_0})-\lmr(f_{t,t_0}) = \log\left(\frac{\intd}{\intd z} 
			G_{t,t_0}(z)\Big|_{z=0} \right)
			=\log\left.\left(\frac{G_{t,t_0}(z)}{z} \right)\right|_{z=0}\\
			&\mbox{}\hspace{0.4cm}= \frac{1}{2\pi i} \int_{\partial D(t,t_0)} 
				\log\left(\frac{G_{t,t_0}(\xi)}{\xi} \right) \frac{\intd \xi}{\xi} = 
				 \frac{1}{2\pi} \int_{\partial D(t,t_0)} \log\left(\frac{G_{t,t_0}(\xi)}{\xi} 
				\right)	\intd \arg \xi\\
			&\mbox{}\hspace{0.4cm}= \frac{1}{2\pi} \int_{\partial D(t,t_0)} \ln 
				\left|\frac{G_{t,t_0}(\xi)}{\xi} \right| \intd \arg \xi,
	\end{align*}
	as $\lmr(f)$ is a real quantity. $|G_{t,t_0}|$ is constant on each concentric slit, so we find
	\[
		\lmr(f_{t_0,t_0})-\lmr(f_{t,t_0}) = \frac{1}{2\pi}  \int_{\partial\D} \ln 
		\left|\frac{G_{t,t_0}(\xi)}{\xi} \right| |\intd \xi|
		=\frac{1}{2\pi}  \int_{s_{t_0,t,t_0}} \ln 
		|G_{t,t_0}(\xi)|\, |\intd \xi|.
	\]
	Next we set  $H_{t,t_0}:=f_{t,t_0}\circ h_t^{-1}$,
	$F_{t,t_0}:=h_{t_0}\circ h_{t}^{-1}$ and $T_{t_0}:=f_{t_0,t_0}\circ h_{t_0}^{-1}$.
	Consequently we have by substitution, the mean value theorem and by using the relation
	$G_{t,t_0}\circ H_{t,t_0} = T_{t_0}\circ F_{t,t_0}$
	\begin{align*}
		&\lmr(f_{t_0,t_0})-\lmr(f_{t,t_0}) = \frac{1}{2\pi}  \int_{\sigma_{t_0,t}}
			|H_{t,t_0}'(\zeta)|\cdot\ln \big|(G _{t,t_0}\circ H_{t,t_0})(\zeta)\big|\, 
			|\intd \zeta|\\
		&\mbox{}\hspace{0.4cm}= \frac{1}{2\pi} |H_{t,t_0}'(\zeta_{t,t_0})|
			\int_{\sigma_{t_0,t}} \ln \big|( T_{t_0}\circ F_{t,t_0})(\zeta)\big|\, 
			|\intd \zeta|
	\end{align*}
	for some $\zeta_{t,t_0}\in \sigma_{t_0,t}.$ Since $\zeta_{t,t_0}\rightarrow \zeta(t_0)$, we find 
	$|H'_{t,t_0}(\zeta_{t,t_0})|\rightarrow \alpha(t_0)$ as $t\searrow t_0$ by 
	Lemma \ref{Lem-ContinuityAlpha}. 
	
	Moreover, the function $\tilde T_{t_0}(z):= \frac{1}{\xi(t_0)}\cdot T_{t_0}(\zeta(t_0)z)$
	is a mapping that fulfills the conditions of Lemma \ref{Lem-Inequality}, so we find for every
	$\delta>0$ an $\epsilon>0$, so that
	\[
		|z|^{c+\delta} \le |\tilde T_{t_0}(z)| \le |z|^{c-\delta}
	\]
	holds for all $z\in B_\epsilon(1)\cap \D$, where $c=\tilde T_{t_0}'(1)>0$. Note that
	$c=|T_{t_0}'(\zeta(t_0))|=\alpha(t_0)$. As a consequence of $|\xi(t_0)|=|\zeta(t_0)|=1$ 
	we get 
	\[
		|z|^{c+\delta} \le | T_{t_0}(z)| \le |z|^{c-\delta}
	\]
	for all $z\in B_{\epsilon}(\zeta(t_0))$. On top of this, if $t$ is close enough
	to $t_0$ we get $F_{t,t_0}(\zeta)\in B_\epsilon(\zeta(t_0))$ for all $\zeta\in\sigma_{t_0,t}$.
	Thus we have for all $t\in(t_0,t_0+\rho)$ where $\rho(\delta)>0$ is small
	\begin{align*}
		\frac{1}{2\pi}& |H_{t,t_0}'(\zeta_{t,t_0})| (\alpha(t_0)+\delta) \int_{\sigma_{t_0,t}}
			\ln|F_{t,t_0}(\zeta)|\,|\intd\zeta|\\
		&\le \lmr(f_{t_0,t_0})-\lmr(f_{t,t_0}) 
		\le \frac{1}{2\pi} |H_{t,t_0}'(\zeta_{t,t_0})| (\alpha(t_0)-\delta) \int_{\sigma_{t_0,t}}
			\ln|F_{t,t_0}(\zeta)|\,|\intd\zeta|.
	\end{align*}
	Moreover in the same way as before we can see that
	\[
		\frac{1}{2\pi}\int_{\sigma_{t_0,t}} \ln|F_{t,t_0}(\zeta)|\,|\intd\zeta| = 
			\lmr(h_{t_0})-\lmr(h_{t}).
	\]
	By combining this with the previous inequality we get for all $t\in(t_0,t_0+\rho),$
	\[
		|H_{t,t_0}'(\zeta_{t,t_0})| (\alpha(t_0)-\delta)
		\le \frac{ \lmr(f_{t_0,t_0})-\lmr(f_{t,t_0})}{\lmr(h_{t_0})-\lmr(h_{t})}
		\le |H_{t,t_0}'(\zeta_{t,t_0})| (\alpha(t_0)+\delta).
	\]
	As $\delta>0$ is arbitrary, we get in the limit case the existence of $\lambda(t_0)$ if and
	only if $\mu(t_0)$ exists. Moreover we find
	\[
		\lambda(t_0) = \alpha(t_0)^2 \cdot \mu(t_0)
	\]
	as $ |H_{t,t_0}'(\zeta_{t,t_0})|$ tends to $\alpha(t_0)$ by Lemma \ref{Lem-ContinuityAlpha}, 
	so the proof is complete.
	
	The other case $t\nearrow t_0$, i.e.
	\[
		\lim_{t\nearrow t_0} \frac{\lmr(f_{k;t,t_0})-\lmr(f_{k;t_0,t_0})}{t-t_0}
		= \alpha^2_k(t_0)\cdot \lim_{t\nearrow t_0} \frac{\lmr(h_{k;t})-\lmr(h_{k;t_0})}{t-t_0}
	\]
	follows in the same way.
\end{proof}
\begin{lemma} \label{Lem-LambdaSimplyConnectedCase}
	Let $\Omega$ be simply connected, i.e. $\Omega=\D$. Then $\alpha_k(t)\le 1$ 
	for all $t\in[0,T]$.
\end{lemma}
\begin{proof}
	First, let $0\le\lt<\gt\le T$, $0\le\ltau<\gtau\le T$
	and $A:=f_{k;\lt,\ltau}(\gamma_k[\lt,\gt])$, $B:=f_{k;\lt,\ltau}(\bigcup_{j\ne k} 
	\gamma_j[\ltau,\gtau])$. By using the chain rule we get
	\begin{align*}
		\lmr(A)&=\lmr(f_{k;\gt,\ltau})-\lmr(f_{k;\lt,\ltau}),\quad
		\lmr(B)=\lmr(f_{k;\lt,\gtau})-\lmr(f_{k;\lt,\ltau}),\\
		&\lmr(A\cup B)=\lmr(f_{k;\gt,\gtau})-\lmr(f_{k;\lt,\ltau}).
	\end{align*}
	Furthermore, as $\Omega$ is simply connected, we have the following inequality 
	(see \cite{MR0155992}):
	\[
		\lmr(A\cup B) \le \lmr(A) + \lmr(B).
	\]
	By combining this inequality with the previous equations we obtain
	\begin{align} \label{Equ-lmrInequality}
		\lmr(f_{k;\gt,\gtau})-\lmr(f_{k;\lt,\gtau}) \le \lmr(f_{k;\gt,\ltau})-\lmr(f_{k;\lt,\ltau}).
	\end{align}
	Next we find together with Lemma \ref{Lem-ConnectionMuLambda}
	\[
		\alpha_k^2(t_0) = \lim_{t\to t_0} \frac{\lmr(f_{k;t,t_0})-\lmr(f_{k;t_0,t_0})}
			{\lmr(h_{k;t})-\lmr(h_{k;t_0})} =
			\lim_{t\searrow t_0} \frac{\lmr(f_{k;t,t_0})-\lmr(f_{k;t_0,t_0})}
			{\lmr(f_{k;t,0})-\lmr(f_{k;t_0,0})} \le 1.
	\]
\end{proof}
\begin{proof}[Proof of Theorem \ref{The-KLEvsLE}]
	This follows immediately from Lemma \ref{Lem-KLE-Lambda1}, Lemma \ref{Lem-KLE-Lambda2},
	Lemma \ref{Lem-ContinuityAlpha} and \ref{Lem-ConnectionMuLambda}.
\end{proof}
\begin{proof}[Proof of Proposition \ref{Cor-Disjoint2}]
	\begin{selflist}
	\item First, we find a 
		unique continuous function $v_2\colon[0,L]\rightarrow [0,T]$ so that 
		$\lmr(\tilde g_s)=s$ since $\lmr(h_{1,v_1(s)})=u_1(s)\leq Ks<s$.
		Note that the continuity is an immediate consequence of Proposition 7
		from \cite{BoehmLauf}. Consequently it remains to prove that $v_2$ is bijective. First we note that it is
		clear that $v_2([0,L])=[0,T]$, so it remains to show that $v_2$
		is injective.
	
	    Let $0\le s_1 < s_2\le L$
		and assume $v_2(s_1)=v_2(s_2).$ We denote by $f_{t,\tau}\colon\D\setminus(\gamma_1[0,t]\cup\gamma_2[0,\tau])\rightarrow \D$
		the normalized Riemann map from $\Omega\setminus(\gamma_1[0,t]\cup\gamma_2[0,\tau])$ onto
		$\D$.
		By using equation (\ref{Equ-lmrInequality}) with 
		$\lt:=v_1(s_1)$, $\gt:=v_1(s_2)$
		$\ltau:=v_2(s_1)$ and $\gtau:=v_2(s_2)$ we obtain
		\begin{align*}
			s_2- s_1&=\lmr(f_{\gt,\gtau})-\lmr(f_{\lt,\ltau}) =
				\lmr(f_{\gt,\ltau})-\lmr(f_{\lt,\ltau})\\
				&\le \lmr(f_{\gt,0})-\lmr(f_{\lt,0}) = \lmr(h_{1;\gt})-\lmr(h_{1;\lt})
				<s_2-s_1.
		\end{align*}
		This is a contradiction, so $v_2$ needs to be bijective. Note that this argumentation
		does not use the fact that $\gamma_1(0)\not=\gamma_2(0)$.
	\item Now we suppose that $u_1$ is continuously differentiable and prove that equation \eqref{Equ-KLE-everywhere} holds.
	First we set $\tilde\gamma_1(s):=(\gamma_1\circ v_1)(s)$,  $\tilde\gamma_2(s):=(\gamma_2\circ v_2)(s)$ and denote by
	 $\tilde f_{s_1,s_2}:\D\setminus(\tilde\gamma_1[0,s_1]\cup\tilde\gamma_2[0,s_2])\rightarrow \D$
		the normalized Riemann map from $\D\setminus(\tilde\gamma_1[0,s_1]
		\cup\tilde\gamma_2[0,s_2])$ onto $\D$.
		Let be $Z=\{0,\ldots,s_N\}$ a partition of the interval $[0,s]$ and
		\[
			S_1(s,Z):=\sum_{l=0}^{N-1} \lmr(\tilde{f}_{s_{l+1},s_l})- \lmr(\tilde{f}_{s_l,s_l}),\quad
			S_2(s,Z):=\sum_{l=0}^{N-1} \lmr(\tilde{f}_{s_l,s_{l+1}})- \lmr(\tilde{f}_{s_l,s_l}).
		\]
		Since $\lmr(\tilde{g}_s)=s$ for all $s\in[0,L],$ by Proposition 17 from \cite{BoehmLauf} the limits 
		$c_k(s):=\lim_{|Z|\to0}S_k(s,Z)$ exist and form increasing and Lipschitz continuous functions $s\mapsto c_k(s),$
		with $c_1(s)+c_2(s)=s$ for all $s\in[0,L].$ On the one hand, again by Proposition 17 from \cite{BoehmLauf}, the limits
		$$\tilde{\lambda}_k(s)=\lim_{t\to s}\frac{\lmr(\tilde{f}_{k;t,s})-\lmr(\tilde{f}_{k;s,s})}{t-s}$$
		exist and coincide with $\dot{c}_k(s)$ for every point $s\in[0,L]$ at which $c_k$ is differentiable. On the other hand,
		according to Lemmas \ref{Lem-ContinuityAlpha} and \ref{Lem-ConnectionMuLambda}, the continuous differentiability of
		$u_1$ implies that $s\mapsto \tilde{\lambda}_1(s)$ is continuous on $[0,L]$. Therefore, $c_1$ and hence $c_2(s)=s-c_1(s)$
		are, in fact, continuously differentiable. It follows that $s\mapsto \tilde{\lambda}_2(s)$ is also continuous and that
		$\tilde{\lambda}_1(s)+\tilde{\lambda}_2(s)=1$ for all $s\in[0,L].$ Now it remains to apply Theorem 2 from \cite{BoehmLauf}
		to conclude that $\tilde{g}_s$ satisfies equation \eqref{Equ-KLE-everywhere} for all $s\in[0,L].$
		
\end{selflist}
\end{proof}

	\section{Proof of Theorem \ref{counterexample} and Theorem \ref{lines}} 
\label{Cha-LocalBehaviour}

In this section we prove Theorems \ref{counterexample} and \ref{lines}. We will
use a different setting, namely the upper half-plane and the chordal Loewner equation, 
instead of the radial case in the unit disk. Here, the role of the logarithmic
mapping radius is played by the so called half-plane capacity, which has nicer 
properties for our purpose. First, we describe the chordal Loewner equation and prove the
chordal analogs of Theorems \ref{counterexample} and \ref{lines}. At the end of this chapter
we justify why it makes sense to consider this different setting.\\

Denote by $\Ha:=\{z\in\C\;|\;\Im(z)>0\}$ the upper half-plane. A bounded subset 
$A\subset\Ha$ is called a \emph{(compact) hull} if $A=\Ha\cap \overline{A}$ and 
$\Ha\setminus A$ is simply connected. By $g_A$ we denote the unique conformal
mapping from $\Ha\setminus A$ onto $\Ha$ with \emph{hydrodynamic normalization}, 
i.e. 
\begin{equation}\label{def_hcap}
	g_A(z)=z+\frac{b}{z}+\LandauO(|z|^{-2}) \quad
	\text{for} \quad |z|\to\infty
\end{equation}	
and for some $b\geq0$. The quantity $\hcap(A):=b$ is called \emph{half-plane capacity}
of $A.$ We note four important properties of $\hcap;$ see \cite{Lawler:2005}, p. 69 and p. 71.
\begin{lemma}\label{hcap}${}$
 
\begin{enumerate}
 \item[a)] $\hcap(c\cdot A)=c^2\cdot \hcap(A)$ for every $c>0$ and every hull $A.$
 \item[b)] If $A_1, A_2$ are two hulls such that $A_1\cup A_2$ is also a hull, then
 $$\hcap(A_1\cup A_2)\leq \hcap(A_1)+\hcap(A_2).$$  This inequality is strict if
 both hulls are nonempty.
 \item[c)] If $A_1, A_2$ are two hulls such that $A_1\cup A_2$ is a hull as well, then 
 $\hcap(A_1)\geq \hcap(g_{A_2}(A_1)).$ If both hulls are nonempty, then the inequality is strict.
 \item[d)]  If $A_1, A_2$ are hulls with $A_1\subset A_2$, then 
 $\hcap(A_2)-\hcap(A_1)= \hcap(g_{A_1}(A_2\setminus A_1)).$
\end{enumerate}
\end{lemma}

If $\gamma\colon[0,T]\to \overline{\Ha}$ is a simple curve, i.e. a continuous, one-to-one 
function with $\gamma(0)\in\R$ and $\gamma((0,T])\subset \Ha,$ then we call the hull 
$\Gamma:=\gamma((0,T])$ a \emph{slit}. If the function $t\mapsto b(t):=\hcap(\gamma((0,t]))$ is  
differentiable at $t_0$, then the family $g_t:=g_{\gamma(0,t]},$ $0\leq t \leq T,$ satisfies the 
following \emph{chordal Loewner equation} (see \cite{Lawler:2005}, Chapter 5):
\begin{equation}\label{ivp}
   \dot{g}_{t_0}(z)=\frac{\dot{b}(t_0)}{g_{t_0}(z)-U(t_0)},
\end{equation}

where $U(t_0)=g_{t_0}(\gamma(t_0))$.\\
$\gamma$ is called \emph{half-plane parametrization} of $\Gamma$ if 
$\hcap(\gamma(0,t])=t$ for all $t\in[0,T].$\footnote{Sometimes (e.g. in \cite{Lawler:2005}, p. 93), a parametrization $\gamma$ is called 
half-plane parametrization if $\hcap(\gamma(0,t])=2t$ for all $t\in[0,T].$ The reason is explained in
\cite{Lawler:2005}, p. 99.}\\

Furthermore, we will need the following definition:\\
Let $\varphi\in(0,\pi).$ We say that $\Gamma$ \emph{approaches $\R$ at
$x\in\R$ in $\varphi$-direction} if for every $\eps>0$ there is a $t_0>0$ such that
$\gamma(0,t_0]$ is contained in the set 
$\{z\in\Ha\; | \; \varphi-\eps<\arg(z-x)<\varphi+\eps\}.$\\

We will need the following lemma about half-plane capacities of straight line segments.
\begin{lemma} \label{two_lines}
Let $b_1, b_2>0$ and let $\Gamma_1, \Gamma_2$ be two line segments starting at $0$ 
with angles $\alpha_1, \alpha_2 \in(0,\pi)$,
$\alpha_1< \alpha_2,$ and $\hcap(\Gamma_1)=b_1, \hcap(\Gamma_2)=b_2.$ Then 
$$\hcap(\Gamma_1\cup \Gamma_2)\to b_1+b_2$$
as $(\alpha_1, \alpha_2)\to(0,\pi).$
\end{lemma}
\begin{proof}
  Let $\gamma_j:[0,b_j]\to \Gamma_j$ be the half-plane parametrization of $\Gamma_j$,
  i.e. $\hcap(\gamma_j(0,t])=t.$\\
	
	We will use a formula which translates the half--plane capacity of an arbitrary hull $A$ 
	into an expected value of a random variable derived from a Brownian motion hitting this hull.
	Let $B_s$ be a Brownian motion started in $z\in \Ha\setminus A.$ We write 
	$\operatorname{\bf P}^{z}$ and $\operatorname{\bf E}^{z}$ for probabilities and 
	expectations derived from $B_s.$ Let $\tau_A$ be the smallest time $s$ with 
	$B_s\in \R\cup A.$ Then formula (3.6) of Proposition 3.41 in \cite{Lawler:2005} tells us
	\[
		\hcap(A)=\lim_{y\to\infty}y\operatorname{\bf E}^{yi}[\operatorname{Im}(B_{\tau_A})].
	\]
	Let $\varrho=\tau_{\Gamma_1}$ and $\sigma=\tau_{\Gamma_2}.$ Then we have 
	(compare with the proof of Proposition 3.42 in \cite{Lawler:2005})
	\begin{eqnarray*} \hcap(\Gamma_1)+\hcap(\Gamma_2)-\hcap(\Gamma_1\cup \Gamma_2)=\\\lim_{y\to \infty}y 
		\left(\operatorname{\bf E}^{yi}[\operatorname{Im}(B_{\sigma});\sigma>\varrho]+ 
		\operatorname{\bf E}^{yi}[\operatorname{Im}(B_{\varrho});\sigma<\varrho]\right).
	\end{eqnarray*}
Here we use the notation $\operatorname{\bf E}^{z}[X;A] := \operatorname{\bf E}^{z}[X {\bold 1}_A]$,
 where $X$ is a random variable and ${\bold 1}_A$ is the indicator function of the event $A$.\\
In the following we will estimate the term $\operatorname{\bf E}^{yi}[
(B_{\sigma});\sigma>\varrho]$, assuming that $y$ is so large that $yi$ is not contained
in the union of the two slits.\\
 
First we note that $\gamma_j(1)$ and $\Im(\gamma_j(1))$ can be computed explicitly; 
see Example 3.39 in \cite{Lawler:2005}:
\begin{equation}\label{tip}
	\gamma_j(1)=\sqrt{2}\cdot (\sqrt{\alpha_j/\pi})^{2\alpha_j/\pi-1}
	\cdot (\sqrt{1-\alpha_j/\pi})^{1-2\alpha_j/\pi} e^{i\alpha_j} \cdot \sqrt{b_j}
\end{equation}
and consequently
\begin{equation*}
	\Im(\gamma_1(1))=\sin(\alpha_j)\cdot \sqrt{2}\cdot (\sqrt{\alpha_j/\pi})^{2\alpha_j/\pi-1}
	\cdot (\sqrt{1-\alpha_j/\pi})^{1-2\alpha_j/\pi}\cdot \sqrt{b_j}.
\end{equation*}

Note that $\Im(\gamma_j(1))\to 0$ and $|\gamma_j(1)|\to \infty$ as $\alpha_j\to0$ or
$\alpha_j\to\pi$. \\
Let $R>0$ and assume that $\alpha_1$ is so close to $0$ that $\Im(\gamma_1(1))<R$ and
\begin{equation}\label{(*)}\tag{$*$}
 |\gamma_1(1)|>R 
\end{equation}
 and write
\[
	\operatorname{\bf E}^{yi}[\operatorname{Im}(B_{\sigma});\sigma>\varrho] = 
	\operatorname{\bf E}^{yi}[\operatorname{Im}(B_{\sigma});\sigma>\varrho\wedge |B_\varrho|<R] 
	+\operatorname{\bf E}^{yi}[\operatorname{Im}(B_{\sigma});\sigma>\varrho\wedge 
	|B_\varrho|\geq R].
\]

{\bf The first summand:} We have $\operatorname{\bf E}^{yi}[\operatorname{Im}(B_{\sigma});
\sigma>\varrho\wedge |B_\varrho|<R] \leq \Im(\gamma_2(1))\cdot 
\operatorname{P}\{B_\varrho\in \Gamma_1\cap\{|z|<R\}\}.$
Now we use that the limit $\lim_{y\to\infty} y\operatorname{\bf P}\{B_\varrho\in 
\Gamma_1\cap\{|z|<R\}\}$ exists; see \cite{Lawler:2005}, p. 74; and that there exists a 
universal constant $c_2$ such that 
\[
	\lim_{y\to\infty} y\operatorname{\bf P}\{B_\varrho\in \Gamma_1\cap\{|z|<R\}\} 
	\leq c_2 \diam(\Gamma_1\cap\{|z|<R\})=c_2\cdot R;
\]	
see \cite{Lawler:2005}, p. 74.
Thus we get
\[	
	\lim_{y\to\infty} y\operatorname{\bf E}^{yi}[\operatorname{Im}(B_{\sigma});
	\sigma>\varrho\wedge |B_\varrho|<R] \leq c_2 R \cdot  \Im(\gamma_2(1))\to 0 
	\quad \text{as} \quad (\alpha_1,\alpha_2)\to(0,\pi).
\]

{\bf The second summand:} First we have 
$\operatorname{\bf E}^{yi}[\operatorname{Im}(B_{\sigma});\sigma>\varrho\wedge  
|B_\varrho|\geq R] \leq \Im(\gamma_2(1))\cdot \operatorname{\bf P}^{yi}\{B_\sigma 
\in \Gamma_2;\sigma>\varrho\wedge  |B_\varrho|\geq R\}.$

A Brownian motion satisfying $\sigma>\varrho\wedge  |B_\varrho|\geq R$ will hit $\Gamma_1$ 
at a point $Q$ with $|Q|\geq R$ and afterward it has to hit $\Gamma_2$ without hitting the 
real axis. Call the probability of this event $p_Q.$\\
From \eqref{(*)} it follows that the Brownian motion hitting $Q$ has to leave 
the half-disk $\{z\in\Ha\cup\R \;|\; |z-\Re(Q)|<R\}$ without hitting the real axis; 
see Figure \ref{Brownian_motion}. From Beurling's estimate (Theorem 3.76 in \cite{Lawler:2005}) it follows that 
$p_Q \leq c_1\cdot \Im (Q) \leq c_1 \cdot \Im(\gamma_1(1)).$
\footnote{Note that Theorem 3.76 in \cite{Lawler:2005} gives an estimate on the probability
that a Brownian motion started in $\D$ will not have hit a fixed curve, say $[0,1]$,
when leaving $\D$ for the first time. The estimate we use can be simply recovered by 
mapping the half-circle $\D\cap \Ha$ conformally onto $\D\setminus [0,1]$
by $z\mapsto z^2.$}
So we get 
\[
	\operatorname{\bf P}^{yi}\{B_{\sigma}\in \Gamma_2;\sigma>\varrho\wedge |B_\sigma|\geq R\} 
	\leq  \operatorname{\bf P}^{yi}\{B_{\sigma}\in \Gamma_2\} \cdot c_1\cdot \Im(\gamma_1(1)). 
\]
Again we have $ \lim_{y\to\infty} y\operatorname{\bf P}^{yi}\{B_\sigma\in \Gamma_2\}\leq c_2 
\diam(\Gamma_2)=c_2\cdot |\gamma_2(1)|.$\\

Thus, using \eqref{tip}, we have
\begin{eqnarray*}
	\lim_{y\to\infty} y\operatorname{\bf E}^{yi}[\operatorname{Im}(B_{\sigma});
	\sigma>\varrho\wedge  |B_\varrho|\geq R] \leq \Im(\gamma_2(1))\cdot c_2\cdot |\gamma_2(1)|
	\cdot c_1\cdot \Im(\gamma_1(1))=\\
	c_1c_2 \Im(\gamma_1(1)) \sin(\alpha_2)\cdot |\gamma_2(1)|^2 =2c_1c_2b_1\cdot
	\Im(\gamma_1(1)) \cdot \sin(\alpha_2)
	\cdot   (1-\alpha_2/\pi)^{1-2\alpha_2/\pi}\cdot (\alpha_2/\pi)^{2\alpha_2/\pi-1}.
\end{eqnarray*}

Note that 
\[
	\Im(\gamma_1(1))\to 0, \quad \sin(\alpha_2) \cdot (1-\alpha_2/\pi)^{1-2\alpha_2/\pi}\to \pi, \quad 
	(\alpha_2/\pi)^{2\alpha_2/\pi-1}\to 1  
\]	
and consequently $\lim_{y\to\infty} y\operatorname{\bf E}^{yi}[\operatorname{Im}(B_{\sigma});
	\sigma>\varrho\wedge  |B_\varrho|\geq R] \to 0$
	as $(\alpha_1,\alpha_2)\to (0,\pi)$.\\

In the same way we obtain $\lim_{y\to \infty}y \operatorname{\bf E}^{yi}[\operatorname{Im}
(B_{\varrho});\sigma<\varrho]\to 0$ as $(\alpha_1,\alpha_2)\to (0,\pi)$ and thus
\[
	\hcap(\Gamma_1\cup \Gamma_2)\to \hcap(\Gamma_1)+\hcap(\Gamma_2) \quad \text{as} \quad (\alpha_1,\alpha_2)\to (0,\pi).
\]
\end{proof}
\begin{figure}[h] \label{Brownian_motion}
	\centering
	\includegraphics[width=130mm]{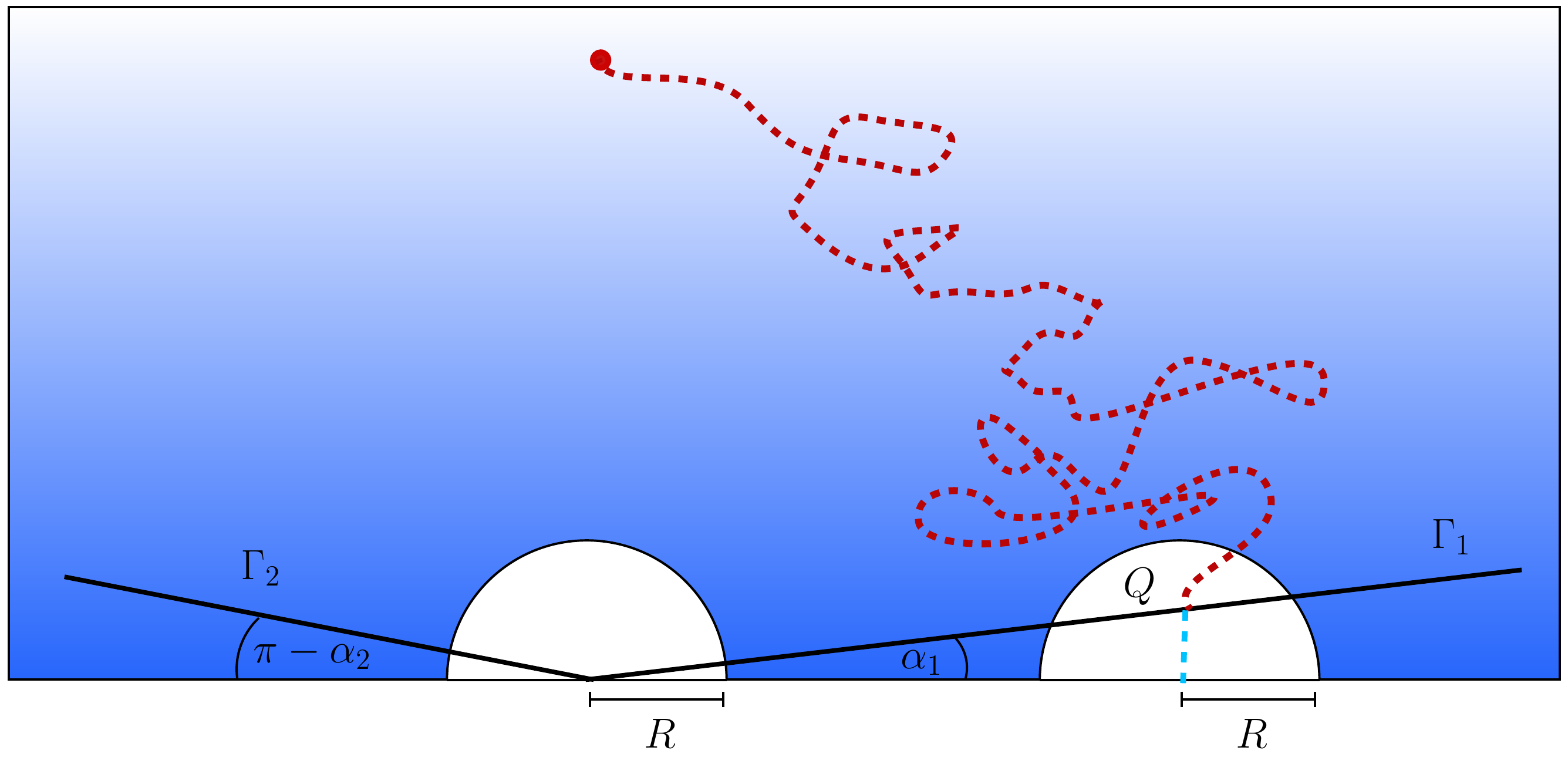}
	\caption{A Brownian motion with $\sigma>\varrho$ and $|B_{\varrho}|\geq R$.}
\end{figure}

Let $\Gamma_1, \Gamma_2$ be two slits with parametrizations $\gamma_1$ and $\gamma_2.$
Furthermore, we let $h_1(t):=\hcap(\gamma_1(0,t])$, $h_2(t):=\hcap(\gamma_2(0,t])$ and
 $c(t):=\hcap(\gamma_1(0,t]\cup\gamma_2(0,t]).$\\
 
\begin{theorem}\label{thm:1}
	Let $b_1,b_2\geq 0$ and let $\Gamma_1, \Gamma_2$ be two slits with $\overline{\Gamma_1} 
	\cap \overline{\Gamma_2}=\{p\}\subset\R,$ such that $\Gamma_j$ approaches
	$p$ in $\alpha_j$-direction for $j=1,2,$ with $0<\alpha_1\leq \alpha_2<\pi.$ 
	Assume that $h_1(t)$ and $h_2(t)$ are differentiable for $t=0$ with $b_1=\dot{h}_1(0),$ 
	$b_2=\dot{h}_2(0).$ Then $c(t)$ is 
	differentiable at $t=0.$
	\begin{itemize}
	 \item[(i)] If $b_1=0$ or $b_2=0,$ then  $\dot{c}(0)=\max\{b_1, b_2\}.$
	\end{itemize}
	If $b_1, b_2>0,$ then
	\begin{itemize}
	 \item[(ii)] $\max\{b_1, b_2\}\leq \dot{c}(0) < b_1+ b_2,$
	 \item[(iii)] $\dot{c}(0)= \max\{b_1, b_2\} \quad \text{if and only if} \quad\alpha_1=\alpha_2  \quad \text{and}$
	 \item[(iv)] $\dot{c}(0)\to b_1+ b_2 \quad \text{as}\quad (\alpha_1,\alpha_2)\to (0,\pi).$
	 \end{itemize}

\end{theorem}
\begin{proof}
	By translation we can assume that $p=0.$\\
	For $t>0,$ we define $G_t=(\gamma_1(0,t]\cup\gamma_2(0,t])/\sqrt{t}.$ By Lemma \ref{hcap} a) we have
	\[
	c(t)/t=\hcap(\gamma_1[0,t]/\sqrt{t}\cup \gamma_1[0,t]/\sqrt{t})=\hcap(G_t).
	\]

	First, we assume that $\Gamma_1$ and $\Gamma_2$ are straight line segments.  
	Since $\hcap(\gamma_j[0,t]/\sqrt{t})=h_j(t)/t\to \dot{h}_j(0)$ as $t\to 0$ for $j=1,2$, we conclude that the tip 
	of the line segment $\gamma_j[0,t]/\sqrt{t}$ converges to the tip of the line segment $L_j$ with the same angle and half-plane capacity 
	$\dot{h}_j(0)=b_j=\hcap(L_j)$. \\
	
	From \cite{LindMR:2010}, Lemma 4.10, it follows that
	$\hcap(G_t)\to \hcap(L_1\cup L_2)$ as $t\to 0.$ Consequently, $c(t)$ is differentiable at $t=0$ with 
	$\dot{c}(0)=\hcap(L_1\cup L_2).$\\
	If $\hcap(L_1)=0$ or $\hcap(L_2)=0,$ then $\hcap(L_1\cup L_2)=\max\{\hcap(L_1), \hcap(L_2)\}.$ 
	This proves (i).\\
	If, on the other hand, $\hcap(L_1),\hcap(L_2)>0,$ then Lemma \ref{hcap} b) gives 
	$$\max\{\hcap(L_1), \hcap(L_2)\}\leq \hcap(L_1\cup L_2)<\hcap(L_1)+\hcap(L_2),$$ hence 
	$\max\{b_1, b_2\}\leq \dot{c}(0) < b_1+b_2.$\\
	We have $\dot{c}(0)=b_j$ if and only if $\hcap(L_j)=\hcap(L_1\cup L_2)$,
	i.e. $L_j=L_1\cup L_2$ which is equivalent to $\alpha_1 = \alpha_2$ and
	$\hcap(L_j)\geq \hcap(L_{3-j})$.\\

	Since $\hcap(L_1\cup L_2)\to \hcap(L_1)+\hcap(L_2) \quad
	\text{as}\quad (\alpha_1,\alpha_2)\to (0,\pi)$ by Lemma \ref{two_lines}, we get
	$\dot{c}(0)\to b_1+ b_2 \quad \text{as} \quad (\alpha_1,\alpha_2)\to (0,\pi).$ Thus, we have shown
	all statements of the theorem for the case of two line segments.\\

Now we pass on to the general case.\\

For $j=1,2$ let $L_j$ be the straight line segment starting at $0$ with angle $\alpha_j$ and 
$\hcap(L_j)=b_j.$ \\
Since $\Gamma_j$ approaches $0$ in $\alpha_j$-direction, we have 
$\Ha \setminus (\gamma_j[0,t]/\sqrt{t}) \to \Ha \setminus L_j$ as $t\to 0$ in the sense of kernel convergence w.r.t. the point
$\infty.$\footnote{Here, $\infty$ is a boundary point of $\Ha$ on the Riemann sphere. However, in our case, 
kernel convergence in $\Ha$ w.r.t. $\infty$ can be defined by  extending the conformal 
mapping $g_A$ analytically to $\C \setminus \overline{A\cup A^*}$, where $A^*$ stands for the reflection of $A$ w.r.t. the real axis.}

From this it follows that $\Ha \setminus G_t \to \Ha \setminus (L_1\cup L_2)$ as $t\to 0$ and, by the definition
of $\hcap$ [see \eqref{def_hcap}] and the Carath\'{e}odory Kernel Convergence Theorem, we obtain 
\[
\hcap(G_t)\to \hcap(L_1\cup L_2)\quad \text{as} \quad t\to 0.
\]
Hence $c(t)$ is differentiable at $t=0$ with $\dot{c}(0)=\hcap(L_1\cup L_2)$.\\
Thus, by using the case of two line segments, we immediately get the statements (i), (ii), (iii) and (iv).
\end{proof}

\begin{theorem}\label{counter_H}
 There exist two slits $\Gamma_1, \Gamma_2,$  with $\overline{\Gamma_1} \cap 
 \overline{\Gamma_2}=\{0\},$ such that $h_j(t)=t$ for all $t\in[0,\hcap(\Gamma_j)]$,
 but $c(t)$ is not differentiable at $t=0.$
\end{theorem}
 \begin{proof}
  Assume that $\Gamma$ is a slit starting at $0$ with half-plane parametrization $\gamma\colon(0,T]\to\C$ having the property
  $\Gamma\subset \{z\in\Ha \,|\, \Re(z)>0\}$ and assume further that
  $\Gamma$ is self-similar in the following sense:
 $$ 1/2\cdot \Gamma \subset \Gamma. $$
 Lemma \ref{hcap} a) implies that $\gamma(0,1/4^n \cdot T]=1/2^n \cdot \Gamma$ 
 for every $n\in \N.$\\
 Now let $\Gamma^*$ be the reflection of $\Gamma$ with
 respect to the imaginary axis, i.e. $\Gamma^*:=\{-\overline{z} \,|\, z\in \Gamma\}.$
 Denote by $\gamma^*$ the half-plane parametrization of $\Gamma^*$ and let
 $K_t=\gamma(0,t]\cup\gamma^*(0,t].$\\ Then also $K_1$ is self-similar,
 i.e. $1/2 \cdot K_t\subset K_t$ and thus for any $t\in[0,T]$ the half-plane capacity 
 $c(t):=\hcap(K_t)$ of the hull $K_t$ satisfies $c(t/4)=c(t)/4$ and consequently
 $$\frac{c(t/4^n)}{t/{4^n}}=\frac{c(t)}{t}$$
 for every $n\in\N.$ Hence, if we assume that $c(t)$ is differentiable at $t=0,$ 
 then $c(t)$ is linear with $c(t)=\dot{c}(0)\cdot t.$\\
 Below we construct such a self-similar slit $\Gamma$ 
 having the property that $c(t)$ is not linear, which gives us the desired contradiction.\\

Let $0\leq\eps<1/2$ and let $A$ be the curve that connects the points $3/4i+\eps/2$, $i+\eps,$ $1/2+i$, $1/2+3/2i$ and $3/2i+\eps$ by straight line segments. Note that $A$ and $1/2\cdot A$ intersect only at $3/4i+\eps/2.$ \\
Now we define the slit $$\Gamma:=\bigcup_{n=0}^\infty 1/2^n\cdot A. $$
Of course, this slit is self-similar, i.e.
$$1/2\cdot \Gamma \subset \Gamma.$$
Let $\Gamma^*$ be the reflection of $\Gamma$ w.r.t. the imaginary axis. 
Now let $\gamma, \gamma^*\colon(0,T]\to \C$  be the parametrizations of
$\Gamma$ and $\Gamma^*$ by half-plane capacity.\\ 
For each $t\in (0,T]$ we can define $K_t$ as the smallest hull containing 
$\gamma(0,t]\cup \gamma^*(0,t].$ Note that $K_t=\gamma(0,t]\cup \gamma^*(0,t]$
for $\eps>0.$ Only for $\eps=0$, the complement of the union has bounded components.
Let $c(t):=\hcap(K_t)$ and let $t_2$ and $t_1$ be defined by 
$\gamma(t_1)=3/4i+\eps/2$ and $\gamma(t_2)=i+\eps.$\\
The quantities $t_2, t_1, c(t_2), c(t_1)$ depend continuously on $\eps,$ as the domains $\Ha\setminus \gamma(0,t],$ 
$\Ha \setminus K_t$ depend continuously on $\eps$ w.r.t. kernel convergence at $\infty$ (see the proof of Theorem \ref{thm:1}).\\

For $\eps=0$ we have $K_{t_{2}}\setminus K_{t_{1}}= \gamma(t_1,t_2]$ and we obtain
\begin{equation*}t_2-t_1\underset{Lemma \ref{hcap} d)}{=}
 \hcap(g_{\gamma(0,t_1]}(\gamma(t_1,t_2]))\underset{Lemma \ref{hcap} c)}{>}\hcap(g_{K_{t_1}}(\gamma(t_1,t_2]))=c(t_2)-c(t_1).
 \end{equation*}

Here, we apply Lemma 18 c) for $A_1=g_{\gamma(0,t_1]}(\gamma(t_1,t_2])$ and 
$A_2 = g_{\gamma(0,t_1]}(K_{t_1}\setminus \gamma(0,t_1]).$ Note that $g_{K_{t_1}} = g_{A_2}\circ g_{\gamma(0,t_1]}.$

Now choose an $\eps>0$ so small that we still have \begin{equation}\label{ltone}\frac{c(t_2)-c(t_1)}{t_2-t_1}<1.\end{equation}
 Assume $c(t)$ is differentiable at $t=0$ in this case. Then $c$ is linear as we have seen before. 
 As $T=\hcap (\Gamma)<c(T)=\dot{c}(0)\cdot \hcap(\Gamma),$ we have $\dot{c}(0)>1.$\\
On the other hand, $\dot{c}(0)<1$ by \eqref{ltone}; a contradiction.
\begin{figure}[h]
	\centering
	\begin{overpic}[width=60mm]
		{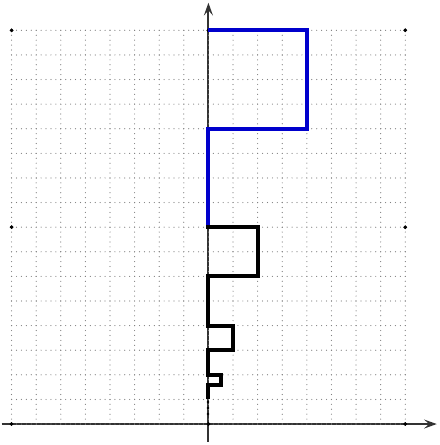}
		\put(70,69){\color{blue}$A$}
		\put(60,47){$\Gamma$}
	\end{overpic}
	\caption{$A$ and $\Gamma$ for $\eps=0.$ }
\end{figure}
\end{proof}

The following lemma gives the connection between the chordal and the radial case that we need
for our purpose. The proof is given in the appendix.

\begin{lemma}\label{chordal_radial_diff}
Let $\gamma_1$ and $\gamma_2$ be the parametrizations of two disjoint slits in a 
circular slit disk $\Omega$ with $\gamma_1(0)=\gamma_2(0)=1.$ 
In the following, $K_t$ is either defined by 
\begin{itemize}
 \item[(i)] $K_t=\gamma_1[0,t]$ for all $t$  or
 \item[(ii)] $K_t=\gamma_1[0,t]\cup\gamma_2[0,t]$ for all $t$.
\end{itemize}

Next, let $g_t$ be the normalized conformal mapping from $\Omega\setminus K_t$ 
onto a circular slit disk.

For $t$ small enough, we can map the hulls into the upper half-plane
$\Ha$ by the mapping $F(z):= -i\log(z)$ (with $\log(1)=0$) and
$A_t:= -i\log(K_t)$ will be a family of increasing $\Ha-$hulls.
Then we have:

$t\mapsto \lmr(g_t)$ is differentiable at $t=0$ if and only if
$t\mapsto \hcap(A_t)$ is differentiable at $t=0.$ 
In this case 
\[
	\frac{d}{dt}\hcap(A_t)(0) = 2\frac{d}{dt}\lmr(g_t)(0).
\]
\end{lemma}

	Now we have all means to prove Theorem \ref{counterexample}
	and Theorem \ref{lines}.
		\begin{proof}[Proof of Theorem \ref{counterexample}]
	 In order to get the desired example in the radial case, we take the two slits 
	 from Theorem \ref{counter_H} and map them, at least locally around 0, into the unit disk
	 by the mapping $z\mapsto e^{iz}.$ This gives us two slits $\Gamma_1$, $\Gamma_2$ in the unit
	 disk with parametrizations $\gamma_1,$ $\gamma_2.$ According to Lemma
	 \ref{chordal_radial_diff}, case (i),
	 $\gamma_1(t)$ and $\gamma_2(t)$ are Loewner parametrizations at $t=0.$\\
	 However, the mapping $t\mapsto g_t$ is not differentiable at $t=0$ because of Lemma 
	 \ref{chordal_radial_diff}, case (ii), and Theorem \ref{Simply_multiply}.
	 
	\end{proof}
	
	\begin{proof}[Proof of Theorem \ref{lines}]
	 
	 Theorem \ref{lines} follows immediately from Theorem \ref{thm:1}, Lemma \ref{chordal_radial_diff}
	 and Theorem \ref{Simply_multiply}.
	 
	\end{proof}
	
\section*{Appendix}\label{appendix}

\begin{proof}[Proof of Lemma \ref{chordal_radial_diff}]
	First of all we set $\Omega_t:=\Omega\setminus K_t$ and $H_t:=\Ha\setminus A_t$
	Then we denote by $h_t:H_t\rightarrow \Ha$ the unique Riemann mapping with hydrodynamic
	normalization. Moreover we set $s_t:=g_t(\partial K_t\cap\partial \Omega_t)
	\subset \partial\D$ and $\tilde s_t:=h_t(\partial A_t\cap \partial H_t)\subset \partial\Ha$.
	Note that $g_t^{-1}$ and $h_t^{-1}$ can be extended continuously to $\partial\D$ and $\partial\Ha$
	by Theorem 2.1 from \cite{pommerenke-boundary}, so we find
	\begin{align*}
		\lmr(g_t) &= -\frac{1}{2\pi} \int_{s_t} \log|g_t^{-1}(\zeta)| |\intd\zeta|,\\
		\hcap(A_t)&= \frac{1}{\pi} \int_{\tilde s_t} \Im(h_t^{-1}(w)) |\intd w|.
	\end{align*}
	A rigorous proof of the first equation can be found in \cite{BoehmLauf}, equation ($\star$), 
	page 12. The second formula can be found, e.g., in \cite{MR1879850}, equation (2.5).\\
	
	If $t$ is small enough, $K_t$ will be close to 
	$1$, i.e for each $\epsilon >0$ we find a $t_0>0$ so that $K_t\subset B_\epsilon(1)$ for 
	all $t\in[0,t_0]$. By Schwarz reflection we see that the function
	\[
		T_t(\zeta):=g_t\big(\exp(i\cdot h_t^{-1}(\zeta))\big)
	\]
	can be extended to a conformal mapping in a small neighborhood around $\tilde s_k$ for all $t\in[0,t_0]$.
	Next we get with $h_t^{-1}(\zeta)=-i\log\big(g_t^{-1}(T_t(\zeta))\big)$ and 
	by usage of the Mean-value Theorem
	\begin{align*}
		\hcap(A_t) &= \frac{1}{\pi}\int_{\tilde s_t} \Im\Big( -i\log\big(g_t^{-1}(T_t(w))
			\big)\Big) |\intd w|
			=-\frac{1}{\pi} \int_{\tilde s_t} \log|g_t(T_t(w))| |\intd w|\\
			&=-\frac{1}{\pi} \int_{s_t} \log |g_t(\zeta)| \frac{1}{|T'_t(T_t^{-1}(\zeta))|}
				|\intd \zeta|= 2\frac{1}{|T_t'(\zeta_t)|} \lmr(K_t),
	\end{align*}
	where $\zeta_t\in \tilde s_t$. Using a normality argument analogous to the proof 
	of Lemma \ref{Lem-ContinuityAlpha},
	$|T_t'(\zeta)|$ tends uniformly to $1$ on a small neighborhood around $0$ as $t\to 0$. Thus 
	$|T_t'(\zeta_t)|\rightarrow 1$ as $t$ tends to zero. 
\end{proof}

\begin{proof}[Proof of Theorem \ref{Simply_multiply} (branch point case)]${}$
		\begin{selflist} 
			\item[a) ]Let $s_t:=g_t(\gamma_1[0,t]\cup\gamma_2[0,t])$ and $F_t:=h_t\circ g_t^{-1}$.
				Then equation ($\star$) on page 12 from \cite{BoehmLauf} gives us
				\[
					\lmr(g_t)= -\frac{1}{2\pi}\int_{s_t}\log|g_t^{-1}(\zeta)| |\intd\zeta| =
					 -\frac{1}{2\pi}\int_{s_t}\log|h_t^{-1}(F_t(\zeta))| |\intd\zeta|.
				\]
				Next we write $\tilde s_t := h_t(\gamma_1[0,t]\cup\gamma_2[0,t])$.
				Each $F_t$ can be extended analytically to $s_t$, so 
				an easy substitution combined with the Mean-value Theorem shows that
				\[
					\lmr(g_t)= -\frac{1}{2\pi}\int_{\tilde s_t}\log|h_t^{-1}(w)|
					\frac{1}{|F'_t(F_t^{-1}(w))|} |\intd w|
					=\frac{1}{|F'_t(\zeta_t)|}\lmr(h_t).
				\]
				Herein $\zeta_t\in s_t$. Finally $s_t$ tends to $\gamma_1(0)$ and $F_t$ can be 
				extended to an analytic function on $B_\epsilon(\gamma_1(0))$ for all $t$ small 
				enough and a small $\epsilon>0$. Consequently $F'_t(\zeta_t)\rightarrow 1$ as 
				$F_t$ tends uniformly to the identical mapping on $B_\epsilon(\zeta_0)$.
			\item[b) ]By using the same methods as in Lemma 10 from \cite{BoehmLauf} we get
				\[
					\log\frac{g_t^{-1}(z)}{z}=\frac{1}{2\pi}\int_{s_t} \log|g_t^{-1}(\zeta)|
					\Phi(\zeta,z;D_t) |\intd\zeta|.
				\]
				Substituting $z=g_t(w)$ in the above equality and using the Mean-value Theorem, we get
				\[
					\log\frac{g_0(w)}{g_t(w)} = \frac{1}{2\pi} \Phi(\zeta_t,g_t(w),D_t)
					\int_{ s_t}\log|g_t^{-1}(\zeta)| |\intd \zeta|= 
					- \Phi(\zeta_t,g_t(w),D_t) \lmr(g_t)
				\]
				with $\zeta_t\in s_t$. Hereby, the continuity of $\Phi$ follows from Lemma 19 from
				\cite{BoehmLauf}. Moreover this Lemma gives 
				$\Phi(\zeta_t,g_t(w),D_t)\rightarrow \Phi(\gamma_1(0),w,D_0)$ 
				as $t$ tends to 0, so the family 
				$t\mapsto g_t$ is differentiable at 0 iff $t\mapsto \lmr(g_t)$ is differentiable.
		\end{selflist}
		\mbox{}\\[-0.9\baselineskip]
		Summarized part a) proves (3.)$\Leftrightarrow$(4.), 
		part b) proves (1.)$\Leftrightarrow$(4.) and 
		part b) applied to $\Omega=\D$ proves (2.)$\Leftrightarrow$(3.).
	\end{proof}

\section*{Acknowledgment}

We would like to thank the referee for the time and effort taken to review our manuscript and
for the many helpful and constructive comments.

\providecommand{\bysame}{\leavevmode\hbox to3em{\hrulefill}\thinspace}
\providecommand{\MR}{\relax\ifhmode\unskip\space\fi MR }
\providecommand{\MRhref}[2]{%
  \href{http://www.ams.org/mathscinet-getitem?mr=#1}{#2}
}
\providecommand{\href}[2]{#2}

\end{document}